\documentclass{amsart}
\usepackage{amsmath,amssymb,amsthm,color,geometry,graphics}
\usepackage[all]{xy}
\DeclareMathOperator{\Ker}{{Ker}}
\DeclareMathOperator{\Coker}{{Coker}}
\DeclareMathOperator{\Hom}{{Hom}}
\DeclareMathOperator{\sHom}{\underline{\Hom}}
\DeclareMathOperator{\RHom}{\mathrm{R}\!\Hom}
\DeclareMathOperator{\End}{{End}}

\DeclareMathOperator{\Ext}{{Ext}}
\DeclareMathOperator{\pd}{proj.dim}
\DeclareMathOperator{\id}{inj.dim}
\DeclareMathOperator{\gd}{gl.dim}
\DeclareMathOperator{\md}{{mod}}
\DeclareMathOperator{\Md}{{Mod}}
\DeclareMathOperator{\smd}{\underline{\md}}
\DeclareMathOperator{\injsmd}{\overline{\md}}
\DeclareMathOperator{\proj}{{proj}}
\DeclareMathOperator{\inj}{{inj}}
\DeclareMathOperator{\thick}{thick}
\DeclareMathOperator{\per}{per}
\DeclareMathOperator{\add}{{add}}
\DeclareMathOperator{\Tr}{Tr}
\DeclareMathOperator{\CM}{{CM}}
\DeclareMathOperator{\sCM}{\underline{\CM}}
\DeclareMathOperator{\GP}{CM}
\DeclareMathOperator{\sGP}{\underline{\GP}}
\DeclareMathOperator{\Supp}{Supp}
\DeclareMathOperator{\gr}{gr}

\DeclareMathOperator{\image}{Im}
\renewcommand{\Im}{\image}

\def\A{\mathcal{A}}
\def\C{\mathcal{C}}
\def\D{\mathcal{D}}
\def\E{\mathcal{E}}
\def\sE{\underline{\E}}
\def\H{\mathcal{H}}
\def\M{\mathcal{M}}
\def\S{\mathcal{S}}
\def\T{\mathcal{T}}
\def\U{\mathcal{Y}}
\def\sU{\underline{\U}}
\def\X{\mathcal{X}}
\def\Z{\mathbb{Z}}

\def\a{\alpha}
\def\b{\beta}

\def\Del{\Delta}
\def\L{\Lambda}
\def\Om{\Omega}
\def\op{\mathrm{op}}
\def\G{\Gamma}
\def\sG{\underline{\Gamma}}

\def\ac{\mathrm{ac}}
\def\rD{\mathrm{D^b}}
\def\Dsg{\mathrm{D_{sg}}}
\def\grsg{\mathrm{D_{sg}^\Z}}
\def\rK{\mathrm{K^b}}
\def\rC{\mathrm{C^b}}
\def\DF{\mathrm{DF}}
\def\KF{\mathrm{KF}}
\def\CF{\mathrm{CF}}

\def\GrF{\mathrm{GrF}}

\newtheorem{Thm}{Theorem}[section]
\newtheorem{Lem}[Thm]{Lemma}
\newtheorem{Prop}[Thm]{Proposition}

\newtheorem{Cor}[Thm]{Corollary}

\theoremstyle{definition}
\newtheorem{Def}[Thm]{Definition}
\newtheorem{Ex}[Thm]{Example}
\newtheorem{Con}[Thm]{Construction}
\newtheorem{Not}[Thm]{Notation}

\theoremstyle{remark}
\newtheorem{Rem}[Thm]{Remark}
\renewcommand{\theprf}

\makeatletter
\renewcommand{\theequation}{\arabic{section}.\arabic{equation}}
\@addtoreset{equation}{section}
\makeatother

\title{Yoneda algebras and their singularity categories}
\author{Norihiro Hanihara}
\address{Graduate School of Mathematics, Nagoya University, Chikusa-ku, Nagoya, 464-8602, Japan}
\email{m17034e@math.nagoya-u.ac.jp}
\thanks{This work is supported by JSPS KAKENHI Grant Number JP19J21165}

\begin{document}
\begin{abstract}
For a finite dimensional algebra $\L$ of finite representation type and an additive generator $M$ for $\md\L$, we investigate the properties of the Yoneda algebra $\G=\bigoplus_{i \geq 0}\Ext_\L^i(M,M)$. We show that $\G$ is graded coherent and Gorenstein of self-injective dimension at most $1$, and the graded singularity category $\grsg(\G)$ of $\G$ is triangle equivalent to the derived category of the stable Auslander algebra of $\L$.
These results remain valid for representation-infinite algebras. For this we introduce the Yoneda category $\U$ of $\L$ as the additive closure of the shifts of the $\L$-modules in the derived category $\rD(\md\L)$. We show that $\U$ is coherent and Gorenstein of self-injective dimension at most $1$, and the singularity category of $\U$ is triangle equivalent to the derived category $\rD(\md(\smd\L))$ of the stable category $\smd\L$. To give a triangle equivalence, we apply the theory of realization functors. We show that any algebraic triangulated category has an f-category over itself by formulating the filtered derived category of a DG category, which assures the existence of a realization functor.
\end{abstract}
\maketitle
\section{Introduction}
Yoneda algebras form a class of algebras which has long been studied in ring theory and representation theory. They are defined, for a ring $\L$ and a $\L$-module $M$, by $\G=\bigoplus_{i \geq 0}\Ext_\L^i(M,M)$ as an abelian group. It has a structure of a graded ring given by the Yoneda product, that is, the concatenation of Yoneda classes of long exact sequences. Usually, Yoneda algebras are studied for the case $M$ is a semisimple module, for example, in Koszul duality \cite{BGG,BGS,Ke94,Ke01,GM1,GM2}, and cohomology rings of finite groups in modular representation theory \cite{Ben2}. 

The subject of this paper, on the other hand, is to investigate the properties of Yoneda algebras $\G$ arising from an additive generator $M$. Although there are some interesting results in this setup for a construction of derived equivalences \cite{HX}, only few studies on such $\G$ have been done. We will give some fundamental ring theoretic (Theorem \ref{thm1}) and representation theoretic (Theorem \ref{introthm}) properties of these Yoneda algebras. Our fully general results will be summarized in Theorem \ref{introThm}.

Let us state our setup precisely. Throughout we fix a field $k$. 
\begin{Def}\label{eqyon}
Let $\L$ be a finite dimensional $k$-algebra of finite representation type, and $M\in\md\L$ an additive generator. We call
\[ \G=\bigoplus_{i \geq 0}\Ext_\L^i(M,M) \]
the {\it Yoneda algebra of $\L$}, which is uniquely determined by $\L$ up to graded Morita equivalence.
\end{Def}

Clearly $\G$ is finite dimensional if and only if $\L$ has finite global dimension, but in general $\G$ is even far from being Noetherian. On the other hand, we have the following notion which generalizes Noetherian property. 
\begin{Def}
Let $R$ be a graded ring.
\begin{enumerate}
\item We say that $R$ is {\it left {\rm (}resp. right{\rm )} graded coherent} if the category $\md^\Z\!R$ (resp. $\md^\Z\!R^\op$) of finitely presented graded left (resp. right) $R$-modules is abelian. We say that $R$ is {\it graded coherent} if it is left and right graded coherent.
\item A graded coherent ring $R$ is called {\it $d$-Gorenstein} if it has injective dimension at most $d$ in $\md^\Z\!R$ and in $\md^\Z\!R^\op$.
\end{enumerate}
\end{Def}
Clearly, any Noetherian graded ring is graded coherent. The notion of Gorenstein rings is important in ring theory \cite{Ma,EJ} and algebraic geometry \cite{Har}, and our definition is an adaptation to coherent rings. Our first main result shows that our Yoneda algebras enjoy nice ring theoretic and homological properties that classical ones do not.
\begin{Thm}\label{thm1}
The Yoneda algebra $\G$ of $\L$ in Definition \ref{eqyon} is graded coherent and $1$-Gorenstein.
\end{Thm}
The next aim of this paper is to study the representation theory of $\G$. Recall that the (graded) {\it singularity category $\grsg(R)$} of a graded ring $R$ is the Verdier quotient of the bounded derived category $\rD(\md^\Z\!R)$ by the perfect derived category $\per^\Z\!R=\rK(\proj^\Z\!R)$, which was introduced by Buchweitz in \cite{Bu} and rediscovered by Orlov in the context of mirror symmetry \cite{O}.
Over Gorenstein rings, the singularity category has another presentation as the {\it stable category of Cohen-Macaulay modules}. 
\begin{Def}
Let $R$ be a graded coherent Gorenstein ring. A graded module $X$ is {\it Cohen-Macaulay} if it is finitely presented and satisfies $\Ext_R^i(X,R)=0$ for all $i>0$.
\end{Def}
These Cohen-Macaulay modules over coherent Gorenstein rings are sometimes called {\it Gorenstein-projective} (or {\it totally reflexive}) modules \cite{ABr,EJ} which are defined over arbitrary rings. For a graded Gorenstein ring $R$, we denote by $\GP^\Z\!R$ the category of graded Cohen-Macaulay $R$-modules. It is a Frobenius category and hence the stable category $\sGP^\Z\!R$ has a natural structure of a triangulated category \cite{Hap}. Moreover, there exists a canonical equivalence $\sCM^\Z\!R \xrightarrow{\simeq} \grsg(R)$ of triangulated categories \cite{Bu,KV,Ric}, by which we will identify these categories. The studies of Cohen-Macaulay modules over Gorenstein rings have attracted enormous attention \cite{CR,Yo,Si,LW,Iy2} and recently, various Gorenstein algebras have been discovered and their Cohen-Macaulay representation theory is investigated, for example, in \cite{AIR,BIRSc,BIY,DL,GLS,IO,IT,JKS,KST,KMV,KR,Ki2,LZ,MU,MYa,SV,U,Ya}.

The second main result of this paper is the following equivalence between the singularity categories of Yoneda algebras and derived categories.
\begin{Thm}\label{introthm}
Let $\G$ be the Yoneda algebra of $\L$ in Definition \ref{eqyon}, and $\sG$ the stable Auslander algebra of $\L$. Then there exists a triangle equivalence
\[ \xymatrix{ \rD(\md\sG) \ar[r]^-\simeq & \grsg(\G)=\sGP^\Z\!\G}, \]
which takes $\sG$ to $\G_0^\ast[-1]$, and $D\sG$ to $\G_0$. Here, $(-)^\ast=\RHom_\G(-,\G)$ and $D=\Hom_k(-,k)$.
\end{Thm}

In fact, the above results can be generalized for arbitrary finite dimensional algebras $\L$ which are not necessarily representation-finite: we consider the {\it Yoneda category}
\[ \U(\L)= \add \{ A[i] \mid A \in \md\L, \ i \in \Z \} \]
which is the full subcategory of the derived category $\rD(\md\L)$ of $\L$. This is a categorical analogue of the Yoneda algebra of $\L$ (see Proposition \ref{UG}). As in the case of algebras, we have the notion of {coherence}, {Gorenstein} property, and {singularity categories} for categories, and {Cohen-Macaulay} modules over Gorenstein categories, see Definitions \ref{cohdef}, \ref{gordef}, and \ref{cmdef}.
The following is our general main result which contains Theorems \ref{thm1} and \ref{introthm}.
\begin{Thm}\label{introThm}
Let $\L$ be an arbitrary finite dimensional algebra and $\U=\U(\L)$ be the Yoneda category of $\L$.
\begin{enumerate}
\item\label{coherence1} {\rm (Lemma \ref{funct}, Theorem \ref{coherence})} $\U$ is coherent.
\item\label{sid1} {\rm (Theorem \ref{YonIG}, \ref{2IG})} $\U$ is $1$-Gorenstein.
\item\label{tilting0} {\rm (Theorem \ref{tilting}, \ref{tilting2})} There exists a triangle equivalence
\[ \xymatrix{ \rD(\md(\smd\L)) \ar[r]^-\simeq & \Dsg(\U)=\sGP\U}. \]
\end{enumerate}
\end{Thm}
The equivalence (\ref{tilting0}) is given in \cite{Ki} (see also \cite[4.11]{IO}) for the very special case $\L$ is hereditary.
We have two independent strategies to build the triangle equivalence above. In both cases, we depend on the fact that $\Dsg(\U)=\sCM\U$ is an {algebraic} triangulated category, that is, the stable category of a Frobenius category.

The first one, which is valid for the restrictive case $\gd\L$ is finite, is {tilting theory} \cite{AHK}. The existence of a tilting object $U$ in an algebraic triangulated category $\T$ implies a triangle equivalence $\T \simeq \per\End_\T(U)$ under a mild assumption (see Theorem \ref{keller}). We have the following reformulation of Theorem \ref{introthm} for the case $\gd\L<\infty$. Here we state the representation-finite case for simplicity, see Theorem \ref{tilting} for the general version.

\begin{Thm}[Theorem \ref{tilting}]
Let $\L$ be a representation-finite algebra of finite global dimension, and let $\G$ be the Yoneda algebra of $\L$ in Definition \ref{eqyon}. Then, we have the following.
\begin{enumerate}
\item $\G_0\in\grsg(\G)$ is a tilting object. 
\item $\End_{\grsg(\G)}(\G_0)$ is the stable Auslander algebra $\sG$ of $\L$, and has finite global dimension.
\end{enumerate}
Consequently, there exists a triangle equivalence $\grsg(\G) \simeq \per\sG=\rD(\md\sG)$.
\end{Thm}

Note that in general there is {\it no} tilting object in $\grsg(\G)$ if $\gd\L$ is infinite. We therefore need the second strategy which works for arbitrary $\L$, namely {realization functors} \cite{BBD,PV}. Let $\T$ be a triangulated category endowed with a $t$-structure and let $\H$ be its heart, which is an abelian category \cite{BBD}. 
A triangle functor $\rD(\H) \to \T$ extending the inclusion $\H \subset \T$ is called a {\it realization functor}. 
In the appendix we formulate the filtered derived category of a DG category, which assures the existence of a realization functor for algebraic triangulated categories (Corollary \ref{algreal}). 
We refer to \cite{KV,Ke90,ChR,M} for different approaches.

To prove Theorem \ref{introthm}, we first show that the triangulated category $\grsg(\G)$ has a $t$-structure, then describe its heart, and finally prove that the realization functor is an equivalence. The results are given as follows, which explain the equivalence in Theorem \ref{introthm}. Again, we state them for the representation-finite case for simplicity.
\begin{Thm}
Let $\L$ be an arbitrary representation-finite algebra and $\G$ its Yoneda algebra in Definition \ref{eqyon}. 
\begin{enumerate}
\item{\rm (Theorem \ref{t-structure})} Set
\begin{equation*}
\begin{aligned}
	t^{\leq 0}&=\{ X \in \grsg(\G) \mid \Hom_{\grsg(\G)}(X,\G_0[i])=0 \text{ for all } i<0 \},\\
	t^{\geq 0}&=\{ X \in \grsg(\G) \mid \Hom_{\grsg(\G)}(X,\G_0[i])=0 \text{ for all } i>0 \}. 
\end{aligned}
\end{equation*}
Then, $(t^{\leq 0}, t^{\geq 0})$ is a $t$-structure on $\grsg(\G)$.
\item{\rm (Proposition \ref{heart})} The heart of $(t^{\leq 0}, t^{\geq 0})$ is equivalent to $\md\sG$ for the stable Auslander algebra $\sG$ of $\L$.
\item{\rm (Theorem \ref{tilting2})} The realization functor $\rD(\md\sG) \to \grsg(\G)$ is a triangle equivalence.
\end{enumerate}
\end{Thm}
One can see that under the equivalence in (3), the $t$-structure in (1) becomes the standard one on $\rD(\md\sG)$. 

Now we explain the structure of this paper. Section \ref{prelim} is a preparatory section where we recall some basic concepts on functor categories and tilting theory. Sections \ref{finitegldim} and \ref{infinitegldim} are central ingredients. We treat the case $\gd\L<\infty$ in Section \ref{finitegldim}, where tilting theory is applied. The general case is treated in Section \ref{infinitegldim} via realization functors. Note that all the results (except those in Section \ref{3.5}) in Section \ref{finitegldim} follow from those in Section \ref{infinitegldim}, but we give separate proofs since the arguments in Section \ref{finitegldim} are much simpler and also motivate some constructions in Section \ref{infinitegldim}. In Section \ref{hereditary}, we discuss the special case $\gd\L \leq 1$. Finally, we give examples in Section \ref{examples}. In the appendix, we give a concrete description of the filtered derived category of a DG category and give a proof of the existence of a realization functor for algebraic triangulated categories.
\subsection*{Acknowledgement}
The author is grateful to his supervisor Osamu Iyama for valuable suggestions and discussions.
\section{Preliminaries}\label{prelim}
\subsection{Functor categories}
We recall in this subsection the basic concepts on functor categories which will be used throughout this paper. 

Let us first fix some notations and conventions.
The composition $\xrightarrow{ \ f \ }\xrightarrow{ \ g \ }$ of morphisms is denoted by $\xrightarrow{ \ fg \ }$. A {\it module} over a preadditive category $\C$ means a contravariant additive functor from $\C$ to the category of abelian groups. In particular, a module over a ring means a left module. 

When $\C$ is essentially small, we denote by $\Md\C$ the abelian category of modules over $\C$, by $\Hom_\C(-,-)$ the $\Hom$-spaces, and by $\Ext_\C^i(-,-)$ the $\Ext$-spaces in $\Md\C$. A $\C$-module $X$ is {\it finitely presented} (resp. {\it finitely generated}) if there exists an exact sequence $\C(-,B) \to \C(-,A) \to X \to 0$ (resp. $\C(-,A) \to X \to 0$) for some $A, B \in \C$. We denote by $\md\C$ the full subcategory of $\Md\C$ consisting of finitely presented modules.

The following notion, introduced in \cite{AR}, provides a fundamental class of categories over which one can develop the representation theory just like over finite dimensional algebras.
\begin{Def}
Let $k$ be a field and let $\C$ be a $k$-linear, $\Hom$-finite category. We say that $\C$ is a {\it dualizing $k$-variety} if $D=\Hom_k(-,k) \colon \Md\C \leftrightarrow \Md\C^\op$ induces a duality $\md\C \leftrightarrow \md\C^\op$.
\end{Def}

The most basic example of a dualizing variety is the category $\proj\L$ of finitely generated projective modules over a finite dimensional algebra $\L$.
If $\C$ is a dualizing variety, then $\md\C$ is an abelian category with enough projectives and injectives. By Yoneda's lemma, the projective objects are given by the representable functors $\C(-,C)$ for $C \in \C$, and by the duality, the injective objects are given by $D\C(C,-)$ for $C \in \C$.

For $C \in \C$ and a full subcategory $\D \subset \C$, we write
\[ \D(-,C)=\C(-,C)|_\D \in \Md\D, \quad \D(C,-)=\C(C,-)|_\D \in \Md\D^\op \]
the restricted functors, even if $C \not\in \D$.
We recall the following notion from \cite{AS}. 
\begin{Def}
Let $\C$ be an additive category and $\D \subset \C$ a full subcategory.
\begin{enumerate}
\item A morphism $D \to C$ in $\C$ is a {\it right $\D$-approximation} of $C$ if $D \in \D$ and it induces an exact sequence $\D(-,D) \to \D(-,C) \to 0$ in $\Md\D$. We say $\D$ is {\it contravariantly finite} in $\C$ if any $C \in \C$ has a right $\D$-approximation.
\item A morphism $C \to D$ in $\C$ is a {\it left $\D$-approximation} of $C$ if $D \in \D$ and it induces an exact sequence $\D(D,-) \to \D(C,-) \to 0$ in $\Md\D^\op$. We say $\D$ is {\it covariantly finite} in $\C$ if any $C \in \C$ has a left $\D$-approximation.
\item $\D$ is {\it functorially finite} in $\C$ if it is both contravariantly and covariantly finite.
\end{enumerate}
\end{Def}

It is not difficult to see the following, which gives more examples of dualizing varieties.
\begin{Prop}\label{ffsdd}
Any functorially finite subcategory of a dualizing variety is again a dualizing variety.
\end{Prop}

Let us now recall the following general notion.
\begin{Def}\label{cohdef}
A category $\C$ is {\it left {\rm (}resp. right{\rm )} coherent} if it has weak kernels {\rm (}resp. weak cokernels{\rm )}. We say $\C$ is {\it coherent} if it is both left and right coherent. 
\end{Def}
Note that $\C$ is left (resp. right) coherent if and only if the category $\md\C$ (resp. $\md\C^\op$) is abelian \cite{Au66}. Therefore dualizing varieties are coherent.

We can speak of the {global dimension} of a coherent category.
\begin{Def}
The {\it left {\rm(}resp. right{\rm)} global dimension} of a coherent category $\C$ is the maximum integer $n$ such that $\Ext_\C^n(-,-)\neq0$ on $\md\C$ (resp. $\md\C^\op$). It is $\infty$ if there is no such $n$.
\end{Def}
We will often use the following well-known result on abelian categories.
\begin{Prop}[\cite{Au66}]\label{ab2}
Any abelian category has global dimension at most $2$ {\rm(}as a coherent category{\rm)}.
\end{Prop}

We also have the notion of  self-injectivity and Gorensteinness of coherent categories.
\begin{Def}\label{gordef}
Let $\C$ be a coherent category. We say $\C$ is {\it self-injective} (resp. {\it $d$-Gorenstein}) if any representable $\C$-modules and $\C^\op$-modules are injective (resp. have injective dimension at most $d$) in $\md\C$ and in $\md\C^\op$.
\end{Def}
For example, any triangulated category is self-injective (see \cite[4.2]{Kr}). It is a dualizing variety if and only if it has a Serre functor \cite[2.11]{IYo}. We note the following well-known result.
\begin{Prop}[\cite{Hap}]\label{DS}
Let $\L$ be a finite dimensional algebra of finite global dimension. Then, $\nu:=D\L \otimes_\L^L-$ gives a Serre functor on $\rD(\md\L)$.
\end{Prop}

Now let us formulate the Cohen-Macaulay modules over Gorenstein categories.
\begin{Def}\label{cmdef}
Let $\C$ be a Gorenstein category. We say that $X \in \md\C$ is {\it Cohen-Macaulay} if $\Ext_\C^i(X,\C(-,C))=0$ for all $C \in \C$ and $i>0$.
\end{Def}

For a Gorenstein category $\C$, we denote by $\CM\C$ the category of Cohen-Macaulay $\C$-modules. Let us collect some basic properties of $\CM\C$, which can be shown in a similar way as in the case of Noetherian rings.
\begin{Prop}\label{cmc}
Let $\C$ be a $d$-Gorenstein category.
\begin{enumerate}
\item\label{omegad} $X \in \md\C$ is Cohen-Macaulay if and only if $X$ is a $d$-th syzygy.
\item\label{resolv} $\CM\C$ is a resolving subcategory of $\md\C$, that is, any projective module in $\md\C$ belongs to $\CM\C$, and $\CM\C \subset \md\C$ is closed under extensions and kernels of epimorphisms.
\item\label{frob} $\CM\C$ is naturally a Frobenius category.
\item The stable category $\sCM\C$ is naturally a triangulated category, and the inverse of its suspension functor is the syzygy functor $\Om$.
\end{enumerate}
\end{Prop}

We end this subsection by noting the following fact, which shows that the Yoneda category $\U(\L)$ is indeed a categorical analogue of Yoneda algebras.
\begin{Prop}[{See \cite[4.1]{ha}}]\label{UG}
Let $\C$ be a category with an automorphism $F$. Suppose there exists $M \in \C$ such that $\C=\add\{F^iM \mid i \in \Z \}$ and set $\G=\bigoplus_{i \in \Z}\Hom_\C(M,F^iM)$. Then, there exists an equivalence
\[ \C \to \proj^\Z\!\G, \quad C \mapsto \bigoplus_{i \in \Z}\Hom_\C(M,F^iC) \]
such that the action of $F$ on $\C$ commutes with the degree shift $(1)$ on $\proj^\Z\!\G$.
\end{Prop}

\subsection{Tilting theory}
We recall some basic facts on tilting theory, which have played an essential role in giving a triangle equivalence. For a subcategory $\M$ of a triangulated category, we denote by $\thick\M$ the smallest full triangulated subcategory containing $\M$ and closed under isomorphisms and direct summands.
\begin{Def}
Let $\T$ be a triangulated category. A full subcategory $\S$ of $\T$ is a {\it tilting subcategory} if it satisfies the following conditions.
\begin{enumerate}
\item $\Hom_\T(X,Y[i])=0$ for all $X, Y \in \S$ and $i \neq 0$.
\item $\thick\S=\T$.
\end{enumerate}
\end{Def}

The importance of tilting subcategories are suggested by the following result. We say that an additive category is {\it idempotent-complete} if any idempotent morphism has a kernel.
\begin{Thm}[{\cite[4.3]{Ke94}}]\label{keller}
Let $\T$ be an idempotent-complete algebraic triangulated category with a tilting subcategory $\S$. Then, there exists a triangle equivalence 
\[ \T \simeq \rK(\S) \]
\end{Thm}

\section{Yoneda algebras for algebras of finite global dimension}\label{finitegldim}
We first consider Yoneda categories for algebras $\L$ of finite global dimension. Note that all the results in this section (except those in Section \ref{3.5}) are contained in the corresponding results in Section \ref{infinitegldim}, and the reader can skip this section. The aim of this section is to give a simpler proof of Theorem \ref{introThm} for the special case $\gd\L<\infty$.
\subsection{Gorenstein property}
Let $\L$ be a finite dimensional $k$-algebra of finite global dimension and $\D=\rD(\md\L)$. Let
\[ \U(\L)= \add \{ A[i] \mid  A \in \md\L, \ i \in \Z \} \]
be the Yoneda category of $\L$. The aim of this subsection is to prove the following result.
\begin{Thm}\label{YonIG}
Let $\L$ be a finite dimensional algebra of finite global dimension and $\U=\U(\L)$ be the Yoneda category of $\L$. Then $\U$ is a $1$-Gorenstein dualizing variety.
\end{Thm}
We start the proof of Theorem \ref{YonIG} with the following observation which is based on \cite[5.1]{Iy}.
\begin{Lem}\label{funct}
Let $L \in \D$.
\begin{enumerate}
\item\label{lapp} If $L \in \rK(\inj\L)$, then $ZL \to L$ is a right $\U$-approximation of $L$.
\item\label{rapp} If $L \in \rK(\proj\L)$, then $L \to CL$ is a left $\U$-approximation of $L$.
\end{enumerate}
Consequently, $\U \subset \D$ is functorially finite, and hence $\U$ is a dualizing variety.
\end{Lem}
\begin{proof} We show (\ref{lapp}). Since $L \in \rK(\inj\L)$, any morphism to $L$ in $\D$ can be presented by a morphism of complexes. Therefore, the natural map $ZL=\bigoplus_{i \in \Z}Z^iL[-i] \to L$ gives a right $\U$-approximation of $L$. The last assertion follows from Proposition \ref{ffsdd}. \end{proof}

A closer look at the above proof shows the following proposition.
\begin{Prop}\label{fpd}
For any $L \in \D$, the triangle $BL \to ZL \to L \to BL[1]$ induces a projective resolution
\[ \xymatrix{ 0 \ar[r]& \U(-,BL) \ar[r]& \U(-,ZL) \ar[r]& \U(-,L) \ar[r]& 0. } \]
Therefore, $\U(-,L) \in \md\U$ has projective dimension at most $1$.
\end{Prop}
\begin{proof} Since $\gd\L<\infty$, we may assume $L \in \rK(\inj\L)$. By Lemma \ref{funct}(\ref{lapp}), we have a right $\U$-approximation $f \colon ZL \to L$. This is a monomorphism in $\rC(\md\L)$ and there exists an exact sequence
\[ \xymatrix{0 \ar[r] & ZL \ar[r]^f & L \ar[r] & BL[1] \ar[r] & 0 } \]
in $\rC(\md\L)$, which induces a triangle
\[ \xymatrix{BL \ar[r] & ZL \ar[r]^f & L \ar[r] & BL[1] } \]
in $\D$. Now, since $\U$ is stable under $[1]$, the morphism $f[-1] \colon ZL[-1] \to L[-1]$ is also a right $\U$-approximation. Therefore, the above triangle yields a short exact sequence
\[ \xymatrix{0 \ar[r] & \U(-,BL) \ar[r] & \U(-,ZL) \ar[r]^{(-,f)} & \U(-,L) \ar[r] & 0. } \]
Since $ZL, BL \in \U$, this proves the assertion.
\end{proof}

We can now prove the main theorem of this subsection.
\begin{proof}[Proof of Theorem \ref{YonIG}] We have already seen in Lemma \ref{funct} that $\U$ is a dualizing variety. Then the injective modules in $\md\U$ are given by $D\U(A,-)$ for each $A \in \U$, which is isomorphic to $\U(-,\nu A)$ by the Serre duality. It has projective dimension $\leq 1$ by Proposition \ref{fpd}. Similarly, any projective modules have injective dimension at most $1$. 
\end{proof}

\subsection{Cohen-Macaulay modules}
We keep the notations from the previous subsection. Our next aim is to study the category $\CM\U$ of Cohen-Macaulay $\U$-modules. 

In the rest, we denote by $[1]_\T=[1]$ the suspension functor on a triangulated category $\T$. The suspension functor $[1]_\D$ on $\D$ restricts to an automorphism of $\U$, and induces an automorphism of $\CM\U$ and of $\sCM\U$, which is also denoted by $[1]_\D$.

We have the following description of Cohen-Macaulay $\U$-modules.
\begin{Prop}\label{CMcategory}
\begin{enumerate}
\item\label{cm} Let $X \in \Md\U$. Then, $X \in \CM\U$ if and only if there exists a triangle
\[ \xymatrix{ A \ar[r]^f & B\ar[r]^g & C\ar[r]^h & A[1] } \]
such that $X=\Coker(-,g)$ and $A, B, C \in \U$.
\item\label{omega3} We have an isomorphism of functors $[3]_{\sCM\U} \simeq [1]_\D$ on $\sCM\U$.
\end{enumerate}
\end{Prop}

We need the following observation for the proof.
\begin{Lem}\label{pj}
Let $L \in \D$. Then, $\U(-,L) \in \md\U$ is projective if and only if $L \in \U$.
\end{Lem}
\begin{proof} We only have to show the `only if' part. Assume $\U(-,L)$ is projective. Then, the projective resolution $0 \to \U(-,BL) \to \U(-,ZL) \to \U(-,L) \to 0$ given in Proposition \ref{fpd} splits. It follows that $\U(-,BL) \to \U(-,ZL)$ is a split monomorphism in $\md\U$, hence so is $BL \to ZL$ in $\U$ by Yoneda's lemma. Therefore, the triangle $BL \to ZL \to L \to BL[1]$ in $\D$ splits and $L \in \U$.
\end{proof}

\begin{proof}[Proof of Proposition \ref{CMcategory}]
(\ref{cm})  If there exists such a triangle, then $X=\Coker(-,g)$ is finitely presented and is a submodule of a projective $\U$-module $\U(-,A[1])$, therefore $X \in \CM\U$ since $\U$ is $1$-Gorenstein. We next show the `only if' part. Assume $X \in \CM\U$ and let $\U(-,B) \xrightarrow{(-,g)} \U(-,C) \to X \to 0$ be a projective presentation of $X$. Complete $g$ to a triangle $A \xrightarrow{} B \xrightarrow{g} C \xrightarrow{} A[1]$. We want to show that $A \in \U$. Consider the following exact sequence in $\md\U$:
\[ \xymatrix@C=4mm@R=1.5mm{
   \U(-,B[-1])\ar^{(-,g[-1])}[rr]&& \U(-,C[-1]) \ar[dr]\ar[rr] && \U(-,A) \ar[rr]\ar[dr] && \U(-,B) \ar[rr]^{(-,g)} && \U(-,C) \ar[dr] & \\
   &&& Y \ar[ur] && Z \ar[ur] &&&& X, } \]
where $Y=\Coker(-,g[-1])$ and $Z=\Ker(-,g)$. We have $Y, Z \in \CM\U$. Indeed, $Y$ is just $X[-1]_\D$, and $Z$ is the second syzygy of $X$. Therefore, $\U(-,A) \in \CM\U$. Then, $\U(-,A)$ is a Cohen-Macaulay module of finite projective dimension by Proposition \ref{fpd}, thus it has to be projective. We conclude that $A \in \U$ by Lemma \ref{pj}.

(\ref{omega3}) This is immediate by (\ref{cm}), indeed, we see by the above exact sequence that $Y$ is $X[-1]_\D$ as well as the third syzygy of $X$.
\end{proof}

We end this subsection with the following remark.
\begin{Prop}\label{nothered}
Let $\L$ be a finite dimensional algebra of finite global dimension. Then the Yoneda category $\U=\U(\L)$ has finite global dimension if and only if $\L$ is semisimple.
\end{Prop}
\begin{proof}
If $\L$ is semisimple, then clearly $\U$ is semisimple. If $\L$ is non-semisimple, there exists a non-split short exact sequence in $\md\L$, hence a non-split triangle in $\D$ with terms in $\U$. By Proposition \ref{CMcategory}, we get a non-projective Cohen-Macaulay $\U$-module, so $\U$ has infinite global dimension.
\end{proof}

\subsection{A tilting subcategory}\label{tilting theory}
In this subsection, we show that the category $\sCM\U$ has a tilting subcategory, hence is equivalent to its perfect derived category.

A {\it $(\U,\U)$-bimodule} is a module over $\U \otimes_k \U^\op$. We first define a $(\U,\U)$-bimodule $\U_0$ as a categorical analogue of the degree $0$ part of Yoneda algebras. This bimodule will play an important role in the sequel.
\begin{Def}\label{U_zero}
Let $\L$ be an arbitrary finite dimensional algebra. The $(\U,\U)$-bimodule $\U_0$ is given by
\[ \U_0(A,B)=\begin{cases} \Hom_\L(A,B) & (A, B \in \md\L) \\ 0 & (\text{otherwise}) \end{cases} \]
for each indecomposable $A, B \in \U$.
\end{Def}
In particular, we obtain $\U_0(-,A) \in \Md\U$ and $\U_0(A,-) \in \Md\U^\op$ for each $A \in \md\L$.

We fix some more notations. Let $\L$ be an arbitrary finite dimensional algebra. We set
\[ \U_i=(\md\L)[-i] \subset \U. \]
Then we have $\U=\bigvee_{i \in \Z}\U_i$, where the right hand side is the smallest additive category containing the $\U_i$'s.  We put $\U_{\geq i}=\bigvee_{j \geq i}\U_j$ and so on. Note that $\U_0$ is equivalent to the ideal quotient $\U/[\bigvee_{i \neq 0}\U_i]$, and the bimodule $\U_0$ defined above is nothing but the natural bimodule structure induced by the projection $\U \twoheadrightarrow \U_0$.

Now let us state the main result of this subsection.
\begin{Thm}\label{tilting}
Let $\L$ be a finite dimensional algebra of finite global dimension.
\begin{enumerate}
\item\label{inCM} We have $\U_0(-,A) \in \CM\U$ for each $A \in \md\L$.
\item\label{tilt} $\sU_0:=\{ \U_0(-,A) \in \sCM\U \mid A \in \md\L \}$ is a tilting subcategory of $\sCM\U$.
\item\label{stable} The natural identification $\md\L=\U_0$ induces an equivalence $\injsmd\L \simeq \sU_0$ of additive categories.
\end{enumerate}
Consequently, there exists a triangle equivalence $\sCM\U \simeq \rD(\md(\injsmd\L))$.
\end{Thm}
We can summarize the statements in the following diagram.
\begin{equation}\label{eqsumm}
   \xymatrix@R=6mm@C=10mm{
   \md\L \ar@{=}[r]\ar@{>>}[d] & \U_0 \ar@{^{(}->}[r]^-{(1)}\ar@{>>}[d] & \CM\U\ar@{>>}[d] \\
   \injsmd\L \ar@{-->}[r]^-{\simeq}_-{(3)}& \sU_0 \ar@{^{(}->}[r]^-{\text{tilting}}_-{(2)} & \sCM\U. }
\end{equation}

For the proof we fix further notations. The {\it support} of a $\U$-module $X$ is the subcategory 
\[ \Supp X=\add\{ A \in \U \text{: indec} \mid X(A) \neq 0 \}. \]
We say that a $\U$-module $X$ is {\it concentrated in degree $\geq i$} (resp. in degree $i$, and so on) if $\Supp X \subset \U_{\geq i}$ (resp. $\Supp X \subset \U_i$, and so on). We note the following fact.
\begin{Lem}\label{dayone}
Let $\L$ be a finite dimensional algebra of finite global dimension.
\begin{enumerate}
\item\label{sodane} Let $A \in \U$. Then $\U(-,A)$ is concentrated in degree $\geq i$ if and only if $A \in \U_{\geq i}$.
\item\label{desuyone} For any $X \in \md\U$, there exists the unique exact sequence
\begin{equation}\label{eqyare}
\xymatrix{ 0 \ar[r] & X_{\geq i} \ar[r] & X \ar[r] & X_{<i} \ar[r] & 0 }
\end{equation}
in $\md\U$ with $X_{\geq i}$ {\rm (}resp. $X_{<i}${\rm )} concentrated in degree $\geq i$ {\rm (}resp. $<i${\rm )}.
\end{enumerate}
\end{Lem}
\begin{proof} The assertion (\ref{sodane}) is clear. In (\ref{desuyone}), it is clear that the exact sequence (\ref{eqyare}) uniquely exists in $\Md\U$. We have to verify that $X_{\geq i}$ and $X_{<i}$ are finitely presented. Since the functors $(-)_{\geq i}$ and $(-)_{<i}$ on $\Md\U$ are exact, we may assume that $X$ is representable. Let $A \in \U$ and $X=\U(-,A)$. Take an injective resolution $A \to L$ so that $L \in \rK(\inj\L)$. Consider the triangle
\[ \xymatrix{ L_{\geq i} \ar[r]& L \ar[r]& L_{<i} \ar[r]& L_{\geq i}[1] } \]
associated to the standard co-$t$-structure on $\rK(\inj\L)$. It induces an exact sequence
\[ \xymatrix{ \U(-,L_{<i}[-1]) \ar[r]^a& \U(-,L_{\geq i}) \ar[r]& \U(-,L) \ar[r]& \U(-,L_{<i}) \ar[r]^b&\U(-,L_{\geq i}[1]),} \]
which lies in $\md\U$ by Lemma \ref{funct}. Then, since $L_{\geq i}$ is a complex of injective modules with terms in degree $\geq i$, the $\U$-module $\U(-,L_{\geq i})$ is concentrated in degree $\geq i$, and therefore so is $\Coker a$. Similarly we see that $\Ker b$ is concentrated in degree $<i$. Therefore the exact sequence
\[ \xymatrix{ 0 \ar[r]& \Coker a \ar[r]& X \ar[r]& \Ker b \ar[r]& 0} \]
must coincide with (\ref{eqyare}), which shows that $X_{\geq i}, X_{<i} \in \md\U$.
\end{proof}

Note that $X_i:=(X_{\geq i})_{\leq i}=(X_{\leq i})_{\geq i}$ can be regarded as the `degree $i$ part' of $X$. In particular, the functor $\U_0(-,A)$ defined above is nothing but the `degree $0$ part' of $\U(-,A)$.

Now let us start the proof with the following important observation.
\begin{Lem}\label{trick}
Let $\L$ be an arbitrary finite dimensional algebra and let $0 \to A \xrightarrow{f} B \xrightarrow{g} C \to 0$ be an exact sequence in $\md\L$.
\begin{enumerate}
\item\label{Binj} If $B \in \inj\L$, then $\Im(-,f)=\U_0(-,A)$.
\item\label{Bproj} If $B \in \proj\L$, then $\Im(g,-)=\U_0(C,-)$.
\end{enumerate}
\end{Lem}
\begin{proof}
We only prove (\ref{Binj}). The exact sequence in $\md\L$ induces an exact sequence
\[ \xymatrix{ \U(-,C[-1]) \ar[r]^\a & \U(-,A) \ar[r]^{(-,f)} & \U(-,B) } \]
in $\md\U$. Since $\U(-,C[-1])$ is concentrated in degree $\geq 1$ by Lemma \ref{dayone}(\ref{sodane}), so is $\Im\a$. Similarly, since $B$ is injective, $\U(-,B)$ is concentrated in degree $0$, thus so is $\Im (-,f)$. Therefore, the exact sequence
\[ \xymatrix{ 0 \ar[r]& \Im\a \ar[r]& \U(-,A) \ar[r]& \Im(-,f)\ar[r] & 0 } \]
must be the exact sequence in Lemma \ref{dayone}(\ref{desuyone}) for $X=\U(-,A)$ and $i=1$. We then conclude that $\Im(-,f)=\U_0(-,A)$.
\end{proof}

Let us fix one more notation.
\begin{Def}\label{EsE}
Let $\L$ be an arbitrary finite dimensional algebra.
\begin{enumerate}
\item Let $\E$ be the category of short exact sequences in $\md\L$ whose morphisms are triples $(f,g,h)$ such that the diagram
\begin{equation}\label{eqmor}
   \xymatrix{
   0 \ar[r]& A' \ar[r]^a\ar[d]_f & A \ar[r]\ar[d]_g &  A''\ar[r]\ar[d]_h & 0 \\
   0 \ar[r]& B' \ar[r]^b & B\ar[r]& B''\ar[r] & 0 }
\end{equation}
is commutative. 
\item Let $\sE$ be the category of short exact sequences up to homotopy, that is, the ideal quotient of $\E$ by split short exact sequences. 
\end{enumerate}
\end{Def}
Note that a morphism $(f,g,h)$ in $\E$ is null-homotopic if and only if $f$ factors through $A$.

We now note the following easy observation. 
\begin{Lem}\label{routine}
The functor $F \colon \E \to \CM\U$ sending $\xi=(0 \to A' \xrightarrow{a} A \xrightarrow{} A'' \to 0)$ to $\Im\U(-,a)$ induces a fully faithful functor $\underline{F} \colon \sE \to \sCM\U$.
\end{Lem}
\begin{proof}
By Proposition \ref{CMcategory}, $\Im\U(-,a)$ is indeed in $\CM\U$. Also, if $\xi$ is split, then $\Im(-,a)=\U(-,A')$ is projective, so $F$ induces a functor $\underline{F} \colon \sE \to \sCM\U$. Therefore it suffices to show that for a morphism $(f,g,h)$ in $\E$ as in (\ref{eqmor}), the corresponding morphism $f_0 \colon \Im(-,a) \to \Im(-,b)$ factors through a projective module in $\CM\U$ if and only if $(f,g,h)$ is null-homotopic. We put $X=\Im(-,a)$ and $Y=\Im(-,b)$.

Suppose that $(f,g,h)$ null-homotopic. Then $f$ factors through $A$, and as in the diagram below $f_0$ equals the composite $X \hookrightarrow \U(-,A) \to \U(-,B') \twoheadrightarrow Y$.
\[ \xymatrix@R=1.6mm{
    \U(-,A') \ar[rr]\ar[dr]\ar[ddd]_f && \U(-,A) \ar@/^/@{-->}[dddll] \\
    & X \ar[ur]\ar[ddd]^<<<<<{f_0} & \\
    \\
    \U(-,B') \ar[rr]|\hole\ar[dr] && \U(-,B) \\
    & Y \ar[ur] }\]
    
Suppose conversely that $f_0 \colon X \to Y$ in $\CM\U$ factors through a projective $\U$-module. Then, it factors through $\U(-,B') \to Y$. 
\[ \xymatrix@R=3mm{
    \U(-,A''[-1]) \ar[r] & \U(-,A') \ar[rr]\ar[dr]\ar[dd]_f && \U(-,A) \\
    && X \ar[ur]\ar[dd]^<<<<{f_0}\ar@{-->}[dl] & \\
    \U(-,B''[-1]) \ar[r] & \U(-,B') \ar[rr]|\hole\ar[dr] && \U(-,B) \\
    && Y \ar[ur] } \]
Using the fact that $\Hom_{\rD(\L)}(A',B''[-1])=0$, we see that the triangle above the dashed line is commutative. Now, since $X \to \U(-,A)$ is an inflation in $\CM\U$ and $\U(-,B')$ is a projective $\U$-module, $X \to \U(-,B')$ factors through $\U(-,A)$, and therefore, $f$ factors through $A$.
\end{proof}

Now we are ready to prove the theorem.
\begin{proof}[Proof of Theorem \ref{tilting}]
(\ref{inCM})  Consider the exact sequence
\[ \xymatrix{ 0 \ar[r] & A \ar[r]^f & I \ar[r] & B \ar[r] & 0 }\]
in $\md\L$ with $I \in \inj\L$. Then we have $\Im(-,f)=\U_0(-,A)$ by Lemma \ref{trick}(\ref{Binj}). Therefore it is in $\CM\U$ by Proposition \ref{CMcategory}.\\
(\ref{tilt})  We have to show the following two statements:
\begin{enumerate}
\renewcommand{\theenumi}{(\alph{enumi})}
\renewcommand{\labelenumi}{\theenumi}
\item\label{ext} $\sHom_\U(\U_0(-,A),\U_0(-,B)[i])=0$ for all $A, B \in \md\L$ and $i \neq 0$.
\item\label{gen} $\thick \sU_0=\sCM\U$.
\end{enumerate}

The claim \ref{ext} follows as in \cite[3.4]{Ya}: since the syzygies of $\U_0(-,A)$ are concentrated in degree $\geq 1$, the supports of $\Om^{\geq1}\U_0(-,A)$ and $\U_0(-,B)$ are disjoint, thus there is no nonzero homomorphism of $\U$-modules between them. This shows the vanishing of the extensions.

Now we turn to \ref{gen}. Note that since $[3]_{\sCM\U}=[1]_\D$ by Proposition \ref{CMcategory}, $\U_0(-,A)[i]_\D$ lies in the thick subcategory of $\sCM\U$ generated by $\sU_0$ for all $A \in \md\L$ and $i \in \Z$.
Consider $\sCM\U$ as the singularity category $\rD(\U)/\per\U$. We show that $\sU_0$ generates $\sCM\U$ by showing that the subcategory $\widetilde{\U_0}:=\{ \U_0(-,A)[i]_\D \in \md\U \mid A \in \md\L, \ i \in \Z \}$ of $\md\U$ generates $\rD(\U)$. 
By $\gd\L<\infty$, any $\U$-module is concentrated in bounded degrees, so by Lemma \ref{dayone}(\ref{desuyone}), it suffices to show that any $\U$-module concentrated in a certain degree lies in the thick subcategory of $\rD(\U)$ generated by $\widetilde{\U_0}$. But such modules can be view as $\U_0$-modules, so we have the assertion since $\gd\U_0 \leq 2$ by Proposition \ref{ab2}.\\
(\ref{stable})  Consider the left square of the diagram (\ref{eqsumm}). It suffices to show that for each $f \colon A \to B$ in $\md\L$, the corresponding morphism $f_0 \colon \U_0(-,A) \to \U_0(-,B)$ factors through a projective module in $\md\U$ if and only if $f$ factors through an injective module in $\md\L$. 
Consider as in (\ref{inCM}) the exact sequences $\a=(0 \to A \xrightarrow{a} I \to A' \to 0)$ and $\b=(0 \to B \xrightarrow{b} J \to B' \to 0)$ in $\md\L$ with $I,J$ being injective. By Lemma \ref{trick}(\ref{Binj}), we have $\Im(-,a)=\U_0(-,A)$ and $\Im(-,b)=\U_0(-,B)$. Note that for each $f \colon A \to B$ in $\md\L$, there is the unique morphism $(f,g,h) \colon \a \to \b$ in $\sE$ and that $f$ factors through an injective $\L$-module if and only if it factors through $I$ if and only if $(f,g,h)$ is null-homotopic. Therefore, our assertion is a consequence of Lemma \ref{routine}.

The rest of the statement follows from Theorem \ref{keller} and the fact that if $\gd\L$ is finite, then so is $\gd\smd\L$ \cite[10.2]{AR}.
\end{proof}

\subsection{Veronese subrings of Yoneda algebras}\label{3.5}
We note the following interesting property of Veronese subalgebras of Yoneda algebras. It turns out, contrary to Theorem \ref{YonIG}, that they have finite global dimension.
\begin{Thm}\label{omakeyon}
Let $\L$ be a finite dimensional algebra with $\gd\L=d \geq 2$. For $2 \leq l \leq d$, consider the subcategory
\[ \U^{(l)}=\add\{ A[il] \mid A \in \md\L, \ i \in \Z \} \]
of $\rD(\md\L)$. Then, the category $\U^{(l)}$ has global dimension at most $3d-1$.
\end{Thm}
Note that this is not true for $l=1$; in this case in fact, $\U^{(l)}=\U$ has infinite global dimension by Proposition \ref{nothered}.

If $\L$ is of finite representation type, we have the following reformulation in terms of algebras.
\begin{Cor}
Let $\L$ be a finite dimensional algebra of global dimension $d \geq 2$, and of finite representation type with an additive generator $M$ for $\md\L$. For $2 \leq l \leq d$, set
\[ \G^{(l)}=\bigoplus_{i \geq 0}\Ext_\L^{il}(M,M). \]
Then, $\gd\G^{(l)} \leq 3d-1$.
\end{Cor}

We start with the following lemma on homological dimensions of $\U^{(l)}$-modules.
\begin{Lem}\label{pdinduction}
Let $A \in \md\L$ and $0<i<l$. Then, $\U^{(l)}(-,A[i]) \in \md\U^{(l)}$ has projective dimension at most $3\cdot\id A-1$.
\end{Lem}
\begin{proof}
We show by induction on $\id A$. If $A$ is injective, then $\U^{(l)}(-,A[i])=0$ and the assertion is clear. Assume $\id A>0$ and consider the exact sequence
\[ \xymatrix{ 0 \ar[r] & A \ar[r] & I \ar[r] & B \ar[r] & 0 } \]
with $I \in \inj\L$. Then we have $\id B=\id A-1$. For $1<i<l$, we have $\U^{(l)}(-,A[i]) \simeq \U^{(l)}(-,B[i-1])$, hence the assertion by the induction hypothesis. It remains to consider $\U^{(l)}(-,A[1])$. By the short exact sequence above, we have a long exact sequence
\[ \xymatrix@C=4mm{
    0 \ar[r] & \U^{(l)}(-,B[-1]) \ar[r] & \U^{(l)}(-,A) \ar[r] & \U^{(l)}(-,I) \ar[r] & \U^{(l)}(-,B) \ar[r] & \U^{(l)}(-,A[1]) \ar[r] & 0 } \]
since $\U^{(l)}(-,I[\pm1])=0$ by $l \geq 2$. Now, since $\pd\U^{(l)}(-,B[-1])=\pd\U^{(l)}(-,B[l-1])\leq3\cdot\id B-1$ by the induction hypothesis, we see that $\pd\U^{(l)}(-,A[1])\leq3+\pd\U^{(l)}(-,B[-1])\leq3\cdot \id A -1$.
\end{proof}

We fix some notations as in Definition \ref{U_zero}.
Define the `degree $0$ part' $\U^{(l)}_0(-,A)$ of $\U^{(l)}(-,A)$ for each $A \in \md\L$ by 
\[ \U^{(l)}_0(B,A)=\begin{cases} \Hom_\L(B,A) & (B \in \md\L) \\ 0 & (B \not\in \md\L) \end{cases} \]
for each indecomposable $B \in \U^{(l)}$. These are the projective modules when considered as $(\md\L)$-modules.

\begin{Lem}\label{pdDelta0}
For each non-injective $A \in \md\L$, we have $\pd\U^{(l)}_0(-,A) \leq 3\cdot\id A-3$.
\end{Lem}
\begin{proof}
Consider as above the exact sequence
\[ \xymatrix{ 0 \ar[r] & A \ar[r] & I \ar[r] & B \ar[r] & 0 } \]
with $I \in \inj\L$. Note that $\id B = \id A-1$. This induces an exact sequence
\[ \xymatrix{
    \U^{(l)}(-,I[-1]) \ar[r] & \U^{(l)}(-,B[-1]) \ar[r] & \U^{(l)}(-,A) \ar[r]^f & \U^{(l)}(-,I). } \]
Then, as in the proof of Lemma \ref{trick}(\ref{Binj}), we have $\Im f=\U^{(l)}_0(-,A)$. Since $\U^{(l)}(-,I[-1])=0$ and $\pd\U^{(l)}(-,B[-1]) \leq 3\cdot(\id A-1)-1=3\cdot\id A-4$ by Lemma \ref{pdinduction}, we obtain $\pd\U^{(l)}_0(-,A) \leq 3\cdot\id A-3$.
\end{proof}
\begin{Rem}\label{pddel0}
If $A$ is injective, then $\U^{(l)}_0(-,A)=\U^{(l)}(-,A)$ is projective.
\end{Rem}

Now we are ready to prove Theorem \ref{omakeyon}.
\begin{proof}[Proof of Theorem \ref{omakeyon}]
We have to show that any $\U^{(l)}$-module $X$ has projective dimension at most $3d-1$. It suffices to consider the case $X$ is concentrated in a certain degree, which we may assume to be $0$. Note that such modules can be viewed as a $(\md\L)$-module. Therefore, for $X \in \md\U^{(l)}$ concentrated in degree $0$, there exists an exact sequence $0 \to P_2 \to P_1 \to P_0 \to X \to 0$ with $P_i \in \add\{ \U^{(l)}_0(-,A) \mid A \in \md\L \}$ since $\gd\md\L \leq 2$ by Proposition \ref{ab2}. By Lemma \ref{pdDelta0} and Remark \ref{pddel0}, we have $\pd P_i \leq 3d-3$. We therefore deduce that $\pd X \leq 3d-1$.
\end{proof}

\section{Yoneda algebras for arbitrary algebras}\label{infinitegldim}
In this section, contrary to Section \ref{finitegldim}, we discuss Yoneda categories for algebras of possibly infinite global dimension. Let $\L$ be an arbitrary finite dimensional algebra. As before, consider the subcategory 
\[ \U(\L)=\add\{A[i] \mid A \in \md\L, \ i \in \Z \} \]
of $\D=\rD(\md\L)$.
\subsection{Coherence}
We first prove in this subsection that the category $\U(\L)$ is always coherent, generalizing Lemma \ref{funct}. We remark that contrary to Lemma \ref{funct}, $\U(\L)$ is not a dualizing variety if $\gd\L=\infty$.

\begin{Thm}\label{coherence}
Let $\L$ be a finite dimensional algebra, and let $\U=\U(\L)$ be the Yoneda category. Then, we have the following.
\begin{enumerate}
\item\label{Ucontra} $\U$ is functorially finite in $\D$.
\item\label{Lfingen} $\U(-,L) \in \Md\U$ {\rm (}resp. $\U(L,-) \in \Md\U^\op${\rm )} is finitely generated for all $L \in \D$.
\item\label{Ulcoh} $\U$ is coherent and thus $\md\U$ is abelian. 
\end{enumerate}
\end{Thm}

The proof is based on the following observation.
\begin{Lem}\label{inftyrapp}
Let $L \in {\rm K^{+,b}}(\inj\L)$ with $H^iL=0$ for $i\geq n$. Then, $Z^{\leq n}L=\bigoplus_{i \leq n}Z^iL[-i] \to L$ gives a right $\U$-approximation.
\end{Lem}
\begin{proof}
We may assume $n=1$. Setting $A:=B^1L=Z^1L \in \md\L$, we have a triangle 
\begin{equation}\label{eqtriapp}
\xymatrix{
A[-1] \ar[r]& L \ar[r]& L_{\leq 0} \ar[r]& A }
\end{equation}
in $\D$. Indeed, consider the triangle $L_{\geq1} \to L \to L_{\leq0} \to L_{\geq1}[1]$ associated to the standard co-$t$-structure on ${\rm K^{+,b}}(\inj\L)$ and note that $L_{\geq1} \simeq A[-1]$ in $\D$.

Consider the morphism $f \colon B[i] \to L$ in $\D$ for each $B \in \md\L$ and $i \in \Z$. If $i<0$, we see by the triangle (\ref{eqtriapp}) that $f$ factors through $A[-1] \to L$. If $i \geq0$, we see as in Lemma \ref{funct} that $f$ factors through $Z^{\leq0}L \to L$.
\end{proof} 

Now our result is a consequence of this lemma.
\begin{proof}[Proof of Theorem \ref{coherence}] We only show the `left' version, the `right' version is dual.
We have (\ref{Lfingen}) by Lemma \ref{inftyrapp}. The other statements are direct consequences, indeed, (\ref{Ucontra}) is just a reformulation of (\ref{Lfingen}). Also, (\ref{Ucontra}) $\Rightarrow$ (\ref{Ulcoh}) is clear since if $\U \subset \D$ is contravariantly finite, then it has weak kernels.
\end{proof}

\subsection{Gorenstein property}\label{2IwanagaGor}
We give an analogue of Theorem \ref{YonIG} in this subsection. The following result states that the category $\U(\L)$ has self-injective dimension at most $1$.
\begin{Thm}\label{2IG}
Let $\L$ be a finite dimensional algebra and $\U=\U(\L)$ be the Yoneda category of $\L$. Then we have the following.
\begin{enumerate}
\item $\id\U(-,A) \leq 1$ in $\md\U$ for all $A \in \U$.
\item $\id\U(A,-) \leq 1$ in $\md\U^\op$ for all $A \in \U$.
\end{enumerate}
\end{Thm}

We need several lemmas analogous to the previous section.
We first have the following observation generalizing Proposition \ref{fpd}.
\begin{Prop}\label{inftyfpd}
Let $L \in \D$. There exists a triangle $A \to B \to L \to A[1]$ in $\D$ with $A, B \in \U$ which induces a projective resolution
\[ \xymatrix{ 0 \ar[r]& \U(-,A) \ar[r]& \U(-,B) \ar[r]& \U(-,L) \ar[r]& 0.} \]
Therefore, $\U(-,L)$ has projective dimension at most $1$.
\end{Prop}
\begin{proof}
Replacing $L$ by its injective resolution, we may assume $L \in {\mathrm{K^{+,b}}}(\inj\L)$ and $H^{\geq n}L=0$. Then by Lemma \ref{inftyrapp}, $Z^{\leq n}L \to L$ gives a right $\U$-approximation of $L$. This injective map in ${\rm C^+}(\md\L)$ induces an exact sequence $0 \to Z^{\leq n}L \to L \to B' \to 0$ as in the diagram below.
\[ \xymatrix{
    Z^{\leq n}L \colon \quad \ar[r]\ar[d]<-2ex> & Z^{n-1}L \ar[r]^0\ar[d] & Z^nL \ar[r]\ar[d] & 0 \ar[r]\ar[d] & 0 \ar[r]\ar[d] & \\
    L \colon \quad \ar[r]\ar[d]<-2ex> & L^{n-1} \ar[r]^d\ar[d] & L^n \ar[r]^d\ar[d] & L^{n+1} \ar[r]^d\ar@{=}[d] & L^{n+2} \ar[r]\ar@{=}[d] & \\
    B' \colon \quad \ar[r] & B^{n}L \ar[r]^0 & B^{n+1}L \ar[r] & L^{n+1} \ar[r]^d & L^{n+2} \ar[r] & } \]
Note that $B' \simeq (B^{\leq n}L)[1]$ in $\D$, therefore, the above exact sequence induces the triangle
\[ \xymatrix{ B^{\leq n}L \ar[r] & Z^{\leq n}L \ar[r] & L \ar[r] & (B^{\leq n}L)[1]} \]
in $\D$. Then, this is a desired triangle. Indeed, this triangle yields the projective resolution
\[ \xymatrix{ 0 \ar[r] & \U(-,B^{\leq n}L) \ar[r] & \U(-,Z^{\leq n}L) \ar[r] & \U(-,L) \ar[r] & 0 } \]
of $\U(-,L)$ since $(Z^{\leq n}L)[-1] \to L[-1]$ is also a right $\U$-approximation.
\end{proof}

We also note the following analogue of Lemma \ref{pj}, which will be used later.
\begin{Lem}\label{inftypj}
Let $L \in \D$. Then, $\U(-,L)$ is projective if and only if $L \in \U$.
\end{Lem}
\begin{proof}
This is same as Lemma \ref{pj}: if $\U(-,L)$ is projective, the triangle $A \to B \to L \to A[1]$ as well as the induced resolution given in Lemma \ref{inftyfpd} splits, hence $L \in \U$.
\end{proof}

Now we consider the functor
\[ \xymatrix{ (-)^\ast=\Hom_\C(-,\C) \colon \Md\C \ar[r] & \Md\C^\op \ar[l] } \]
sending $X \in \Md\C$ to $C \mapsto \Hom_\C(X,\C(-,C))$. Note that $\C(-,C)^\ast=\C(C,-)$ for each $C \in \C$. This is a left exact functor, whose $i$-th right derived functor is
\[ \xymatrix{ \Ext^i_\C(-,\C) \colon \Md\C \ar[r] & \Md\C^\op \ar[l] } \]
sending $X \in \Md\C$ to $C \mapsto \Ext^i_\C(X,\C(-,C))$. 
In what follows, we consider these functors for $\C=\U$.

For each $L \in \D$, the functor $\U(-,L)^\ast$ maps $A \in \U$ to $\Hom_\U(\U(-,L),\U(-,A))$, so we have a morphism $\a_A \colon \U(L,A) \to \Hom_\U(\U(-,L),\U(-,A))=\U(-,L)^\ast(A), \ f \mapsto (-,f)$ for each $A \in \U$, which gives a morphism
\[ \a \colon \U(L,-) \longrightarrow \U(-,L)^\ast \]
of $\U^\op$-modules.
The following observation is crucial.
\begin{Lem}\label{cohsur}
The morphism $\a \colon \U(L,-) \to \U(-,L)^\ast$ is surjective for all $L \in \D$.
\end{Lem}
\begin{proof} By Lemma \ref{inftyfpd}, we have a triangle $A \to B \to L \to A[1]$ in $\D$ which induces a projective resolution
\[ \xymatrix{ 0 \ar[r]& \U(-,A) \ar[r]& \U(-,B) \ar[r]& \U(-,L) \ar[r]& 0. } \]
Now, applying $(-)^\ast$ to this resolution and comparing with the exact sequence obtained from the triangle, we have a commutative diagram with exact rows
\[ \xymatrix{
   \U(A[1],-) \ar[r]\ar[d] & \U(L,-) \ar[r]\ar[d]_\a & \U(B,-) \ar[r]\ar@{=}[d] & \U(A,-) \ar@{=}[d] \\
   0 \ar[r] & \U(-,L)^\ast \ar[r]& \U(B,-) \ar[r] & \U(A,-),} \]
which shows that $\a$ is surjective.
\end{proof}

\begin{Lem}\label{ABseq}
Let $\C$ be a left coherent category. If $X \to X^{\ast\ast}$ is surjective for all $X \in \md\C^\op$, then $\id\C(-,C) \leq 1$ for all $C \in \C$.
\end{Lem}
\begin{proof}
We denote by $\Tr \colon \smd\C \leftrightarrow \smd\C^\op$ the transpose duality \cite{ABr}. By the Auslander-Bridger sequence
\[ \xymatrix{ 0 \ar[r]& \Ext_\C^1(\Tr X,\C) \ar[r]& X \ar[r]& X^{\ast\ast} \ar[r]& \Ext_\C^2(\Tr X,\C) \ar[r] & 0, } \]
we have $\Ext^2_\C(\Tr X,\C)=0$, and hence the assertion.
\end{proof} 
Now we are ready to prove Theorem \ref{2IG}.
\begin{proof}[Proof of Theorem \ref{2IG}]
We prove the `left' version. Let $X \in \md\U^\op$ and $\U(B,-) \to \U(A,-) \to X \to 0$ be a projective presentation of $X$. Complete $A \to B$ to a triangle $L \to A \to B \to L[1]$ in $\D$ so that we have exact sequences
\[ \xymatrix@C=3mm@R=1mm{
   \U(B,-) \ar[rr] && \U(A,-) \ar[rr]\ar[dr] && \U(L,-) \ar[rr] && \U(B[-1],-) \\
   &&& X \ar[ur] &&& } \]
in $\md\U^\op$ and
\[ \xymatrix@C=3mm@R=1mm{
   \U(-,B[-1]) \ar[rr] && \U(-,L) \ar[rr]\ar[dr] && \U(-,A) \ar[rr] && \U(-,B) \\
   &&& X^\ast \ar[ur] &&& } \]
in $\md\U$, indeed, the image of the middle morphism is $X^\ast$ since it is the kernel of $\U(-,A) \to \U(-,B)$. Applying $(-)^\ast$ to the left half of the second exact sequence and comparing it with the right half of the first one, we have a commutative diagram with exact rows
\[ \xymatrix{
   0 \ar[r]& X \ar[r]\ar[d] & \U(L,-) \ar[r]\ar[d]_\a & \U(B[-1],-) \ar@{=}[d] \\
   0 \ar[r] & X^{\ast\ast} \ar[r] & \U(-,L)^\ast \ar[r] & \U(B[-1],-). } \]
Now the middle vertical map is surjective by Lemma \ref{cohsur}, thus so is the left vertical map. We therefore obtain the result by Lemma \ref{ABseq}.
\end{proof}

\subsection{Cohen-Macaulay modules}
Our next aim is to study the category $\CM\U$ of Cohen-Macaulay modules.  We have the following analogue of Proposition \ref{CMcategory}.
\begin{Prop}\label{GPcategory}
Let $\L$ be a finite dimensional algebra, and let $\U=\U(\L)$ be its Yoneda category.
\begin{enumerate}
\item $X \in \Md\U$ is Cohen-Macaulay if and only if there exists a triangle
\begin{equation}\label{eqtri}
\xymatrix{ A \ar[r]^f & B\ar[r]^g & C\ar[r]^h & A[1] }
\end{equation}
such that $X=\Coker(-,g)$ and $A, B, C \in \U$. 
\item We have an isomorphism of functors $[3]_{\sGP\U} \simeq [1]_\D$ on $\sGP\U$.
\end{enumerate}
\end{Prop}
\begin{proof}
Using Theorem \ref{2IG} and Lemma \ref{inftypj} instead of Theorem \ref{YonIG} and Lemma \ref{pj}, the same proof as in Proposition \ref{CMcategory} applies.
\end{proof}
We also have a generalization of Proposition \ref{nothered} with the same proof.
\begin{Prop}
Let $\L$ be an arbitrary finite dimensional algebra. Then the Yoneda category $\U(\L)$ has finite global dimension if and only if $\L$ is semisimple.
\end{Prop}

We later use the following computation of the duality functor $(-)^\ast$.
\begin{Lem}\label{astcomp}
Let $X \in \CM\U$ and take a triangle (\ref{eqtri}) such that $X=\Im(-,f)$.
Then, we have $\Im(f,-)=X^\ast$ in $\CM\U^\op$.
\[ \xymatrix@C=3mm@R=1mm{
   \U(C,-) \ar[rr] && \U(B,-) \ar[rr]^{(f,-)}\ar[dr] && \U(A,-) \\
   &&& X^\ast\ar[ur] &} \]
\end{Lem}
\begin{proof}
We have an exact sequence
\[ \xymatrix@C=3mm@R=1mm{
   \U(-,A)\ar[dr]\ar[rr]^{(-,f)} && \U(-,B) \ar[rr]&& \U(-,C) \\
   & X\ar[ur] &&& } \]
in $\CM\U$. Applying $(-)^\ast$ to this sequence yields the assertion.
\end{proof}
\subsection{A $t$-structure}
Let $\L$ be a finite dimensional algebra and $\U=\U(\L)$ be its Yoneda category. We show that the category $\sGP\U$ has a natural $t$-structure and describe its heart. This will lead to the construction of a realization functor.

Recall the $(\U, \U)$-bimodule $\U_0$ from Definition \ref{U_zero}. By abuse of notation we often view $\U_0$ as a subcategory $\{\U_0(-,A) \in \GP\U \mid A \in \md\L \}$ and $\{\U_0(A,-) \in \GP\U^\op \mid A \in \md\L \}$, or their images in the stable categories. Our main result of this subsection is the following.
\begin{Thm}\label{t-structure}
Let $\L$ be an arbitrary finite dimensional algebra. Set
\begin{equation*}
\begin{aligned}
t^{\leq 0}&=\{ X \in \sGP\U \mid \sHom_\U(X,\U_0[<\!0])=0 \},\\
t^{\geq 0}&=\{ X \in \sGP\U \mid \sHom_\U(X,\U_0[>\!0])=0 \}. 
\end{aligned}
\end{equation*}
Then, $(t^{\leq 0}, t^{\geq 0})$ forms a $t$-structure on $\sGP\U$. Moreover, its heart is equivalent to $\md(\smd\L)$.
\end{Thm}

\begin{Rem}
The above description of the $t$-structure suggests that $\U_0$ is `injective-like' in the following sense: Let $\Del$ be a finite dimensional algebra, and let $\D=\rD(\md\Del)$ be its derived category. Consider the standard $t$-structure $(\D^{\leq0},\D^{\geq0})$ on $\D$, which is given by
\begin{equation*}
\begin{aligned}
\D^{\leq 0}&=\{ X \in \D \mid \Hom_\D(X,D\Del[<\!0])=0 \},\\
\D^{\geq 0}&=\{ X \in \D \mid \Hom_\D(X,D\Del[>\!0])=0 \}. 
\end{aligned}
\end{equation*}
The description of the subcategories in Theorem \ref{t-structure} are analogous, with $\U_0$ in the place of the injective module $D\Del$.
It turns out that the `projective-like' subcategory is $\U_0^\ast[-1]$. These are justified in Lemma \ref{star}.
\end{Rem}

The plan of our proof is as follows: using Wakamatsu-type lemma (Proposition \ref{Wa}), we show in Lemma \ref{aisle} that $t^{\leq 0}$ is an aisle in $\sGP\U$, that is, $(t^{\leq 0}, (t^{\leq 0})^\perp[1])$ is a $t$-structure, and in Lemma \ref{perp} that $(t^{\leq0})^\perp[1]=t^{\geq 0}$. We therefore obtain the desired result.

Put as usual $t^{\leq n}=t^{\leq 0}[-n]$, $t^{\geq n}=t^{\geq 0}[-n]$ for each $n \in \Z$.
We start the proof with the following immediate consequence of the definition.
\begin{Lem}\label{tojiru}
We have $t^{\leq -1} \subset t^{\leq 0}$ and $t^{\geq 1} \subset t^{\geq 0}$.
\end{Lem}

Let us introduce a piece of notation.
\begin{Not}
For $A \in \U$, we write $A=\bigoplus_{i \in \Z}A_i[-i]$ with $A_i \in \md\L$.
\end{Not}

We note the following observation. 
\begin{Lem}\label{Hom(-,U_0)}
We have $\Hom_\U(\U(-,A[i]),\U_0)=\U_0(A_i,-)$ for each $A \in \U$ and $i \in \Z$ as $\U_0^\op$-modules.
\end{Lem}
\begin{proof} Indeed, using Yoneda's lemma, we see that the $\U_0^\op$-module $\Hom_\U(\U(-,A[i]),\U_0)$ maps $B \in \U_0$ to $\Hom_\U(\U(-,A[i]),\U_0(-,B))=\U_0(A[i],B)=\U_0(A_i,B)$. \end{proof}

Recall from Proposition \ref{GPcategory} that $X \in \GP\U$ comes from a triangle (\ref{eqtri}). If $X$ has no projective direct summand, we may assume that all $f, g, h$ belong to the radical of $\U$. We keep the notations from Section \ref{tilting theory}, for example, $\U_{\leq 0}=(\md\L)[\geq\!0]$ and so on. 

An important step toward proving Theorem \ref{t-structure} is the following description of each subcategories.

\begin{Prop}\label{subcat}
Let $X \in \sGP\U$ be an object coming from a triangle
\[ \xymatrix{ A \ar[r]^f & B \ar[r]^g & C \ar[r]^h & A[1]} \]
in $\D$ with each term in $\U$, each map is a radical map, and $X=\Im(-,f)$.
Then, the following hold.
\begin{enumerate}
\item\label{leq0} $X \in t^{\leq 0}$ if and only if $A, B, C \in \U_{\leq 0}$.
\item\label{geq0} $X \in t^{\geq 0}$ if and only if $A, B, C \in \U_{\geq 0}$.
\end{enumerate}
\end{Prop}
\begin{proof}
We first prove (\ref{geq0}) since this is simpler. We have a projective resolution
\[ \cdots \to \U(-,A[-1]) \to \U(-,B[-1]) \to \U(-,C[-1]) \to \U(-,A) \to X \to 0 \]
of $X$. Applying $\Hom_\U(-,\U_0)$ to this sequence yields a complex
\[ \xymatrix@C=4mm{ {}_\U(\U(-,A),\U_0) \ar[r] & {}_\U(\U(-,C[-1]),\U_0) \ar[r] & {}_\U(\U(-,B[-1]),\U_0) \ar[r] & {}_\U(\U(-,A[-1]),\U_0) \ar[r]& \cdots, } \]
which is isomorphic by Lemma \ref{Hom(-,U_0)} to
\begin{equation}\label{eq1}
\xymatrix{ \U_0(A_0,-) \ar[r]^{p_0} & \U_0(C_{-1},-) \ar[r]^{p_1} & \U_0(B_{-1},-) \ar[r]^{p_2} & \U_0(A_{-1},-) \ar[r]^{p_3}& \cdots. }
\end{equation}
Now, $\sHom_\U(X,\U_0[>\!0])=0$ is equivalent to the acyclicity of this complex (\ref{eq1}). It is clear that if $A,B,C \in \U_{\geq 0}$, then the above complex is acyclic. 
We show the converse. Suppose the complex (\ref{eq1}) is acyclic. Then we see that $\Ker p_i$ is projective (in $\md\U_0^\op$) for every $i \geq 1$ since $\gd\U_0^\op\leq 2$ by Proposition \ref{ab2}, hence $p_i$ is a split epimorphism to its image for each $i \geq 0$. On the other hand, since the maps $p_i$ are the summands of $f$, $g$, and $h$, they are radical maps. Therefore, $p_i$ must be zero for all $i \geq 0$, and we conclude that $A_{\leq -1}=0$, $B_{\leq -1}=0$, and $C_{\leq -1}=0$.

Now we turn to (\ref{leq0}). We similarly have the following exact sequence:
\begin{equation}\label{eq2}
\xymatrix@C=3mm@R=1mm{
   \U(-,A) \ar[rr]\ar[dr] && \U(-,B) \ar[rr]\ar[dr] && \U(-,C) \ar[rr]\ar[dr] && \U(-,A[1]) \ar[rr]\ar[dr] && \\
   & X \ar[ur] && X[1] \ar[ur] && X[2] \ar[ur] &&&. }
\end{equation}
Note that the complex
\[ \xymatrix@C=4mm{ \cdots \ar[r] & {}_\U(\U(-,A[1]),\U_0) \ar[r] & {}_\U(\U(-,C),\U_0) \ar[r] & {}_\U(\U(-,B),\U_0) \ar[r]& {}_\U(\U(-,A),\U_0)} \]
obtained by applying $\Hom_\U(-,\U_0)$ to (\ref{eq2}), which is isomorphic by Lemma \ref{Hom(-,U_0)} to
\begin{equation}\label{eq22}
\xymatrix{ \cdots \ar[r] & \U_0(B_1,-) \ar[r] & \U_0(A_1,-) \ar[r] & \U_0(C_0,-) \ar[r]& \U_0(B_0,-) \ar[r]^{q} & \U_0(A_0,-),}
\end{equation}
computes the spaces $\sHom_\U(X[>\!0],\U_0)$, that is, the cohomologies of (\ref{eq22}) are $\sHom_\U(X[>\!0],\U_0)$. Therefore, the vanishing of the extensions are equivalent to the acyclicity of the complex (\ref{eq22}).

We first prove the `only if' part. Suppose the complex (\ref{eq22}) is acyclic. Then, as in the proof of (\ref{geq0}), it is a minimal projective resolution of $\Coker q$ by the assumption that $f, g$ and $h$ are radical maps. We conclude that $A_{\geq 1}=0$, $B_{\geq 1}=0$, and $C_{\geq 1}=0$ by $\gd\U_0 \leq2$ (Proposition \ref{ab2}).

We next prove the `if' part. Suppose that $A, B, C \in \U_{\leq 0}$. We have to show that the sequence
\[ \xymatrix{0 \ar[r] & \U_0(C_0,-) \ar[r] & \U_0(B_0,-) \ar[r] & \U_0(A_0,-)} \]
is exact. Taking the 0-th cohomology of the triangle $A \to B \to C \to A[1]$ and using $A \in \U_{\leq 0}$, we see that $A_0 \to B_0 \to C_0 \to 0$ is exact in $\md\L$. This yields the desired exactness.
\end{proof}

We give a description of the `heart' $\H=t^{\leq 0} \cap t^{\geq 0}$. Recall from Definition \ref{EsE} that $\sE$ is the category of short exact sequences up to homotopy. We also need the well-known equivalence
\begin{equation}\label{eqG}
G \colon \sE \to \md(\smd\L), \quad (0 \to A \xrightarrow{f} B \xrightarrow{g} C \to 0) \mapsto \Coker{}_\L(-,g).
\end{equation}
\begin{Prop}\label{heart} 
\begin{enumerate}
\item\label{yapparikantan} {The functor $F \colon \sE \to \H$ sending $(0 \to A \xrightarrow{f} B \xrightarrow{g} C \to 0)$ to $\Im\U(-,f)$ is an equivalence.}
\item\label{kokoro} We have equivalences of categories $\H \xleftarrow{F} \sE \xrightarrow{G} \md(\smd\L)$.
\end{enumerate}
\end{Prop}
\begin{proof}
(\ref{yapparikantan})  $F$ is fully faithful by Lemma \ref{routine}. It is dense by Proposition \ref{subcat}.\\
(\ref{kokoro})  This follows from (\ref{yapparikantan}) and (\ref{eqG}).
\end{proof}

We need some more information of $\H$. Recall that there exist dualities
\[ (-)^\ast=\Hom_\U(-,\U) \colon \GP\U \longleftrightarrow \GP\U^\op , \quad \sGP\U \longleftrightarrow \sGP\U^\op\]
which satisfy $\U(-,A)^\ast=\U(A,-)$ for each $A \in \U$. This duality plays a key role in the sequel.
\begin{Lem}\label{star}
Let $A \in \md\L$. Consider the equivalence $H=F \circ G^{-1} \colon \md(\smd\L) \to \H$ in Proposition \ref{heart}.
\begin{enumerate}
\item\label{pjlike} The image of the projective $(\smd\L)$-module $\sHom_\L(-,A)$ under $H$ is $\U_0(A,-)^\ast[-1]$. In particular, we have $\U_0(A,-)^\ast[-1] \in \H$.
\item\label{injlike} The image of the injective $(\smd\L)$-module $D\sHom_\L(A,-)$ under $H$ is $\U_0(-,\tau A)$, where $\tau$ is the AR-translation in $\md\L$.
\end{enumerate}
Consequently, $\H$ has enough projectives $\U_0^\ast[-1]$ and enough injectives $\U_0$.
\end{Lem}
\begin{proof}
We consider the images of certain objects in $\sE$ under $F$ and $G$.

(\ref{pjlike})  Consider the exact sequence $\xi=(0 \to B \to P \to A \to 0) \in \sE$ with $P \in \proj\L$.
Its image in $\md(\smd\L)$ under $F \colon \sE \to \md(\smd\L)$ is the projective module $\sHom_\L(-,A)$.

On the other hand, by Lemma \ref{trick}(\ref{Bproj}), we have an exact sequence
\[ \xymatrix@!C=5mm@R=2mm{
   \U(A,-) \ar[rr]\ar[dr] && \U(P,-) \ar[rr] && \U(B,-) &\\
   & \U_0(A,-) \ar[ur] &&&& } \]
in $\GP\U^\op$. Applying $(-)^\ast$, we have
\[ \xymatrix@!C=5mm@R=2mm{
    \U(-,B) \ar[rr] && \U(-,P) \ar[rr]\ar[dr] && \U(-,A) \\
    &&& \U_0(A,-)^\ast \ar[ur] &&.} \]
Therefore the image of $\xi$ under $G \colon \sE \to \H$ is $\U_0(A,-)^\ast[-1]$.

(\ref{injlike})  Note that by the AR-duality, we have $D\sHom_\L(A,-)=\Ext_\L^1(-,\tau A)$. Consider the exact sequence $(0 \to \tau A \to I \to B \to 0) \in \sE$ with $I \in \inj\L$. We see as above that it is mapped to $\Ext_\L^1(-,\tau A)$ by $\sE \to \md(\smd\L)$, and to $\U_0(-,\tau A)$ by $\sE \to \H$.
\end{proof}

We next show the orthogonality of the subcategories. For a subcategory $\C$ of an additive category $\D$, we denote by
\[ \C^\perp=\{X \in \D \mid \Hom_\D(\C,X)=0 \}, \quad  {}^\perp\C=\{X \in \D \mid \Hom_\D(X,\C)=0 \} \]
the orthogonal categories. Also, we identify the opposite category $\U(\L)^\op$ of $\U(\L)$ with the Yoneda category $\U(\L^\op)$ of $\L^\op$ by the usual duality as in the diagram below.
\[ \xymatrix@R=-1pt{
   D \colon \hspace{-15mm}& \rD(\md\L) \ar[r]^\simeq & \rD(\md\L^\op) \ar[l]\\
   & \rotatebox{90}{$\subset$} & \rotatebox{90}{$\subset$} \\
   & \U(\L) \ar[r]^\simeq & \U(\L^\op)\ar[l] } \]
\begin{Lem}\label{perp}
We have $(t^{\leq -1})^\perp=t^{\geq 0}$.
\end{Lem}
\begin{proof}
We first show the inclusion `$\supset$'. For this we show $\Hom_\U(t^{\leq0},t^{\geq1})=0$. Let $X \in t^{\leq0}$ and $X' \in t^{\geq1}$. By Proposition \ref{subcat}, we can take a triangle $A \xrightarrow{f} B \to C \to A[1]$ with $A \in \U_{\leq 0}$ such that $X\simeq \Im (-,f)$, and a triangle $A' \xrightarrow{f'} B' \to C' \to A'[1]$ with $A' \in \U_{\geq 1}$ such that $X' \simeq \Im (-,f')$. Then, we have $\Hom_\U(X,X') \hookrightarrow \Hom_\U(\U(-,A),X') \twoheadleftarrow \Hom_\U(A,A')=0$, and hence the assertion.
 
We next show `$\subset$'. Suppose $X \in (t^{\leq -1})^\perp$.  By Lemma \ref{star}(\ref{pjlike}) we have $\U_0^\ast \subset t^{\leq -1}$, so $\sHom_\U(\U_0^\ast[\geq\!0],X)=0$. Applying $(-)^\ast$ yields $\sHom_{\U^\op}(X^\ast[-1], \U_0[<\!0])=\sHom_{\U^\op}(X^\ast, \U_0[\leq\!0])=0$. Now, let $A \xrightarrow{f} B \xrightarrow{g} C \xrightarrow{h} A[1]$ be a triangle in $\D$ with terms in $\U$ such that $f, g, h$ are radical maps and $X=\Im(-,f)$. Then by Lemma \ref{astcomp}, $C \xrightarrow{g} B \xrightarrow{f} A \to C[1]$ is a triangle in $\U^\op$ with radical maps and $X^\ast[-1]=\Im(-,g)$. By Proposition \ref{subcat}, we see that $C, B, A \in (\U^\op)_{\leq 0}$. Therefore, we deduce that $A, B, C \in \U_{\geq 0}$ and hence $X \in t^{\geq0}$ again by Proposition \ref{subcat}.
\end{proof}

Now we recall Wakamatsu's lemma for triangulated categories. We say that a subcategory $\C$ of a triangulated category $\T$ is {\it extension-closed} if two end terms of a triangle is in $\C$, then so is the middle term. 
\begin{Prop}\label{Wa}\cite[2.3]{IYo}
Let $\T$ be a Krull-Schmidt triangulated category, and let $\C$ be a subcategory which is closed under $[1]$. Assume $\C$ is contravariantly finite and extension-closed in $\T$. Then $(\C, \C^\perp[1])$ is a $t$-structure on $\T$.
\end{Prop}

We are now at the final step of the proof.
\begin{Lem}\label{aisle}
$t^{\leq0}$ is an aisle in $\sGP\U$, that is, $(t^{\leq0}, (t^{\leq -1})^\perp)$ is a $t$-structure.
\end{Lem}
\begin{proof}
The subcategory $t^{\leq 0}$ is clearly extension-closed, and is closed under $[1]$ by Lemma \ref{tojiru}. Therefore by Proposition \ref{Wa}, it suffices to show that $t^{\leq 0}$ is contravariantly finite in $\sGP\U$.

Let $X \in \sGP\U$ and $A \xrightarrow{f} B \xrightarrow{g} C \to A[1]$ be a triangle in $\D$ with terms in $\U$ such that $X=\Im(-,f)$. Let $A_{\leq 0} \to A \to A_{\geq 1} \to A_{\leq0}[1]$ and $B_{\leq 0} \to B \to B_{\geq 1} \to B_{\leq0}[1]$ be the triangles associated to the standard $t$-structure $(\D^{\leq0}, \D^{\geq 0})$ on $\D$. Note that these triangles split and have terms in $\U$. By \cite[1.1.11]{BBD}, there exists a $3 \times 3$ diagram of triangles
\[ \xymatrix{
   A_{\leq 0} \ar[r]\ar[d]^{f'} & A \ar[r]\ar[d]^f & A_{\geq 1}\ar[r]^0\ar[d] & A_{\leq 0}[1] \ar[d] \\
   B_{\leq 0} \ar[r]\ar[d] & B \ar[r]\ar[d]^g & B_{\geq 1}\ar[r]^0\ar[d] & B_{\leq 0}[1] \ar[d] \\
   D \ar[r]\ar[d] & C \ar[r]\ar[d] & E \ar[r]\ar[d] & D[1] \ar[d] \\
   A_{\leq 0}[1] \ar[r] & A[1] \ar[r] & A_{\geq 1}[1] \ar[r]^0 & A_{\leq 0}[2], } \]
namely, the rows and columns are triangles, each square but the bottom right corner is commutative, and the bottom right square is anti-commutative.

We first claim that $D \in \U$. We know that $D \in \D^{\leq 0}$ and $E \in \D^{\geq 0}$ since each column is a triangle. Since $C=C_{\leq -1} \oplus C_{\geq 0}$ and $\Hom_\D(C_{\leq -1},E)=0$, the triangle $D \to C \to E \to D[1]$ is a direct sum of triangles $D' \to C_{\geq 0} \to E \to D'[1]$ and $C_{\leq -1} \xrightarrow{1} C_{\leq -1} \to 0 \to C_{\leq-1}[1]$, where $D' \oplus C_{\leq -1}=D$. Now, $D' \in \D^{\leq 0}$ since it is a direct summand of $D$, and also $D' \in \D^{\geq 0}$ since it is the mapping cocone of $C_{\geq 0} \to E$. Therefore $D' \in \md\L$, and hence $D=D' \oplus C_{\leq -1} \in \U$.

Set $Y=\Im(-,f')$. By the claim above, we have $Y \in \GP\U$ and by Proposition \ref{subcat}, we see that $Y \in t^{\leq 0}$. Note that $Y$ is a submodule of $X$ and setting $L=\Coker(Y \to X)$, we have a commutative diagram with exact rows
\[ \xymatrix{
    0 \ar[r]& \U(-,A_{\leq0}) \ar[r]\ar@{>>}[d]& \U(-,A) \ar[r]\ar@{>>}[d]& \U(-,A_{\geq1})\ar[r]\ar@{>>}[d]& 0\\
    0 \ar[r]& Y\ar[r]& X \ar[r]& L \ar[r]& 0. }\]

We claim that $Y \to X$ gives a right $t^{\leq 0}$-approximation of $X$ in $\sGP\U$. Let $W \in \GP\U$ whose image in $\sGP\U$ lies in $t^{\leq 0}$, and $W \to X$ be a morphism in $\GP\U$. It suffices to show $\Hom_\U(W,L)=0$. Since $W \in t^{\leq 0}$ in $\sGP\U$, there is a surjection $\U(-,F) \twoheadrightarrow W$ in $\md\U$ with $F \in \U_{\leq 0}$ by Proposition \ref{subcat}. Therefore, we have $\Hom_\U(W,L) \hookrightarrow \Hom_\U(\U(-,F), L) \twoheadleftarrow \Hom_\U(F,A_{\geq 1})=0$, and hence the assertion.
\end{proof}

We have completed the proof of the main result of this subsection. Let us summarize the proof below.
\begin{proof}[Proof of Theorem \ref{t-structure}]
By Lemma \ref{aisle} and \ref{perp}, $(t^{\leq0}, t^{\geq0})$ is certainly a $t$-structure .
The statement on the heart is Proposition \ref{heart}.
\end{proof}

We note the boundedness of the $t$-structure constructed above.
\begin{Prop}\label{bdd}
The $t$-structure $(t^{\leq0}, t^{\geq0})$ given in Theorem \ref{t-structure} is bounded.
\end{Prop}
\begin{proof}
We have to show that for any $X \in \sCM\U$, there exists $n \in \Z$ such that $\sHom_\U(X,\U_0[\leq\!-n])=0$ and $\sHom_\U(X,\U_0[\geq\!n])=0$. Note that if $X \in \md\U$ is concentrated in degree $\geq i$, then $\Om^3X$ is concentrated in degree $\geq i+1$ since $\gd\U_0 \leq 2$ by Proposition \ref{ab2}.
It follows that $\Om^{\geq n}X$ is concentrated in degree $\geq 1$ for sufficiently large $n$, so we have the second statement since $\U_0$ is concentrated in degree $0$.
Also, we have the first statement as in \cite[7.2]{ha}. 
\end{proof}

\subsection{The triangle equivalence}
We are now ready to prove a main result of this paper. Let $\L$ be an arbitrary finite dimensional algebra. Once we know that there is a $t$-structure on $\sGP\U$ with heart $\md(\smd\L)$, there exists by Corollary \ref{algreal}, a realization functor $\rD(\md(\smd\L)) \to \sGP\U$.
\begin{Thm}\label{tilting2}
Let $\L$ be a finite dimensional algebra. The realization functor
\[ F \colon \rD(\md(\smd\L)) \to \sGP\U \]
is a triangle equivalence.
\end{Thm}
\begin{proof}
We first verify the condition in Theorem \ref{realization}(\ref{coeff}) that the realization functor $F$ is fully faithful. Let $X, Y \in \H$, $n \geq 1$ and $X \to Y[n]$ be a morphism in $\sGP\U$. By Lemma \ref{star}, we can take an injection $Y \hookrightarrow I$ in $\H$ with $I \in \U_0$. Then, the composition $X \to Y[n] \to I[n]$ is zero since $I \in \U_0$.
Now, we have $\sGP\U=\thick\H$ by Proposition \ref{bdd}. We therefore conclude that $F$ is an equivalence.
\end{proof}
\section{Yoneda algebras for hereditary algebras}\label{hereditary}
The aim of this section is to discuss in further detail the Yoneda algebras for hereditary algebras. In the first subsection, we note some consequences of the previous sections, and apply these results in the following two subsections.
\subsection{Basic properties}
We first record the self-injectivity of Yoneda algebras in the hereditary case, which also characterizes hereditary algebras.
The proof is based on the following well-known fact. Recall that we denote by $\D$ the derived category $\rD(\md\L)$ of $\L$.
\begin{Prop}[{\cite[I.5.2]{Hap}}]\label{wellknown}
We have $\U=\D$ if and only if $\L$ is hereditary.
\end{Prop}

We immediately obtain the following result.
\begin{Prop}\label{iffhered}
Let $\L$ be an arbitrary finite dimensional algebra and $\U=\U(\L)$ be the Yoneda category of $\L$. Then, $\U$ is self-injective if and only if $\L$ is hereditary.
\end{Prop}
\begin{proof}
If $\L$ is hereditary, then $\U=\D$ by Proposition \ref{wellknown}, thus $\U$ is self-injective. If $\L$ is not hereditary, then there exists $L \in \D\setminus\U$ by Proposition \ref{wellknown}. Then, $\U(-,L) \in \md\U$ has projective dimension precisely $1$ by Lemma \ref{inftyfpd} and \ref{inftypj}, so $\U$ cannot be self-injective.
\end{proof}

From now on, we restrict ourselves to representation-finite case: Let $\L$ be a finite dimensional hereditary algebra of finite representation type with an additive generator $M$ for $\md\L$, and $\G$ be its Yoneda algebra. Then, we have $\G=\End_{\L}(M) \oplus \Ext^1_{\L}(M,M)$.
If $\L$ is the path algebra $kQ$ for a Dynkin quiver $Q$, then the derived category is presented by the infinite translation quiver $\Z Q$ (see \cite{ASS,Hap}) with mesh relations \cite{Hap,Rie}. Thus we have the following explicit description of the Yoneda algebra of $kQ$.
\begin{Prop}\label{waru1}
Let $Q$ be a Dynkin quiver. Then, the Yoneda algebra of the path algebra $kQ$ is presented by the quiver $\Z Q/[1]$ and the mesh relations.
\end{Prop}
Set $\Gamma_0=\End_{\L}(M)$ and $E=\Ext^1_{\L}(M,M)$. Moreover, let $e$ be the idempotent of $\G_0$ corresponding to the maximal injective summand $I$ of $M$. Then, $\sG=(1-e)\G_0 (1-e)=\G_0/\G_0 e \G_0$ is the stable Auslander algebra, and we have natural surjections
\[ \G \twoheadrightarrow \G_0 \twoheadrightarrow \sG.\] 
Note that $Ee=\Ext_{\L}^1(M,I)=0$, so $\G_0 e=\G e$ is a projective $\G$-module and $\G(1-e)=\G$ holds in $\smd \G$.

We record the following special case of Proposition \ref{CMcategory} and Theorem \ref{tilting}, upon which the results of the following subsections are build. Since $\G$ is self-injective by Proposition \ref{iffhered}, we have $\GP^\Z\!\G=\md^\Z\!\G$ in this case. 
\begin{Thm}\label{8.2}
Let $\L$ be a representation-finite hereditary algebra, $\G$ its Yoneda algebra, and $\sG$ the stable Auslander algebra of $\L$.
\begin{enumerate}
\item\label{8.2.1} We have an isomorphism of functors $[3] \simeq (1)$ on $\smd^\Z\!\G$, where $(1)$ is the degree shift.
\item\label{8.2.2} There exists a triangle equivalence
\[ \rD(\md\sG) \simeq \smd^\Z\!\G \]
taking $\sG \in \rD(\md\sG)$ to $\G_0(1-e)=\G_0 \in \smd^\Z\!\G$.
\end{enumerate}
\end{Thm}

\begin{Rem}
The equivalence in (\ref{8.2.2}) can also be deduced using the results of \cite{Ya}.
\end{Rem}

In the following subsections, we give two independent applications of this equivalence.
\subsection{Fractional Calabi-Yau property of stable Auslander algebras}
As a first application of the triangle equivalence in Theorem \ref{8.2}, we deduce that the stable Auslander algebra of a representation-finite hereditary algebra has the fractional Calabi-Yau property. For integers $a$ and $b \neq0$, we say that a triangulated category $\T$ is {\it fractionally Calabi-Yau of dimension $a/b$} ($a/b$-CY for short) if it has a Serre functor $\nu$ such that $\nu^b \simeq [a]$. A finite dimensional algebra $\L$ (of finite global dimension) is fractionally CY if its bounded derived category $\rD(\md\L)$ is fractionally CY.

We first note the well-known result on fractional CY property of representation-finite hereditary algebras.
\begin{Prop}[\cite{MY}]\label{shitteru}
Let $\L$ be a representation-finite hereditary algebra of Dynkin type $\Del$, and let $h$ be its Coxeter number, which is given as in the following table:
\[
\begin{tabular}{|c||*{5}{c|}}
\hline
$\Delta$ & $A_n$ & $D_n$ & $E_6$ & $E_7$ & $E_8$ \\
\hline
$h$ & $n+1$ & $2n-2$ & $12$ & $18$ & $30$ \\
\hline
\end{tabular}
\]
Then, $\L$ is $(h-2)/h$-CY.
\end{Prop}

We have the following result as an application. 
\begin{Thm}\label{frac}
Let $\L$ be a representation-finite hereditary algebra of Dynkin type $\Del$, and let $h$ be its Coxeter number. Then, the stable Auslander algebra of $\L$ is $(2h-6)/h$-{CY}.
\end{Thm}

We need the following observation which is a variant of \cite{AR96}\cite[8.5]{Ke05}.
\begin{Lem}\label{variant}
Let $\T$ be an $a/b$-CY triangulated category. Then, $\smd\T$ is $(3a-b)/b$-CY.
\end{Lem}

Now we are ready to prove our result.
\begin{proof}[Proof of Theorem \ref{frac}]
Note that by $\proj^\Z\!\G \simeq \rD(\L)$, we have $\smd^\Z\!\G \simeq \smd\rD(\L)$. Together with Theorem \ref{8.2}, we deduce $\rD(\sG) \simeq \smd\rD(\L)$. Now the result is a consequence of Proposition \ref{shitteru} and Lemma \ref{variant}.
\end{proof}

\subsection{Rigid modules over Yoneda algebras}
So far we have considered the category of {\it graded} modules over the graded algebra $\G$. In this subsection, we discuss the category $\smd\G$ of {\it ungraded} modules over $\G$. Let us first recall the related notions.
\begin{Def}
Let $\C$ be a full subcategory of a triangulated category $\T$ and $n \geq 1$.
\begin{enumerate}
\item $\C$ is {\it $n$-rigid} if $\Hom_\T(\C,\C[i])=0$ for all $0<i<n$.
\item $\C$ is {\it maximal $n$-rigid} if it is $n$-rigid and any $n$-rigid subcategory containing $\C$ equals $\C$.
\item\cite{Iy07a} $\C$ is {\it $n$-cluster tilting} if it is functorially finite in $\T$ and satisfies the following.
\begin{equation*}
\begin{aligned}
\C&=\{ X \in \T \mid \Hom_\T(\C,X[i])=0 \text{ for all } 0<i<n \} \\
    &=\{ X \in \T \mid \Hom_\T(X,\C[i])=0 \text{ for all } 0<i<n \}.
\end{aligned}
\end{equation*}
\end{enumerate}
\end{Def}

Our main result of this section states the $\smd\G$ is endowed with a maximal $3$-rigid object.
\begin{Thm}\label{ungraded}
\begin{enumerate}
\item\label{5.10.1} We have 
\[ \add\G_0=\{ X \in \smd \G \mid \Ext_\G^i(\G_0,X)=0 \ \text{and} \ \Ext_\G^i(X,\G_0)=0 \ \text{for} \ i=1,2 \}. \]
\item\label{5.10.2} $\G_0 \in \smd \G$ is maximal $3$-rigid.
\item\label{ct} There is no $3$-cluster tilting subcategory in $\smd \G$ containing $\G_0$ unless $\smd \G$ is semisimple.
\end{enumerate}
\end{Thm}

Although the theorem can be stated only in terms of $\smd\G$, our proof is based on the equivalence in Theorem \ref{8.2}.
We start with the following proposition by which we can regard $\md\sG$ as the full subcategory of $\smd \G$.
\begin{Prop}\label{emb}
\begin{enumerate}
\item\label{embed}
The composition 
\[ \md\sG \hookrightarrow \rD(\md\sG) {\simeq} \smd^\Z\!\G \to \smd \G \]
is fully faithful.
\item\label{embe}
The above functor coincides as a map on objects with the composition
\[ \md\sG \hookrightarrow \md\G_0 \hookrightarrow \md \G \to \smd \G \]
induced by the canonical surjections $\G \twoheadrightarrow \G_0 \twoheadrightarrow \sG$. 
\end{enumerate}
\end{Prop}
\begin{proof} 
(\ref{embed})  Let $X,Y \in \md\sG$ and $U,V \in \smd^\Z\!\G$ be the corresponding objects under the equivalence $\rD(\md\sG) \simeq \smd^\Z\!\G$. It suffices to show $\sHom^\Z_\G(U,V(i))=0$ for $i \neq 0$. Indeed, since $(1)=[3]$ on $\smd^\Z\!\G$, we have $\sHom^\Z_\G(U,V(i)) = \Hom_{\rD(\sG)}(X,Y[3i])$, which is zero for $i \neq 0$ by $\gd\sG \leq 2$.\\
(\ref{embe})  This can be seen by observing that both compositions take $\sG$ to $\G_0$, and take short exact sequences to triangles.
\end{proof}

In what follows, we identify $\md\sG$ with the full subcategory of $\smd \G$ via Proposition \ref{emb}.
We note the following consequence of the above proposition.
\begin{Lem}\label{ring}
We have $\sG=\G/\G e \G$. In particular, for $X \in \smd \G$, $X \in \md\sG$ if and only if $X$ is annihilated by $e$.
\end{Lem}
\begin{proof}Since $\sG=\G_0/\G_0 e \G_0$, it is enough to show that $\G e \G$ contains $E$. 
Let $M \hookrightarrow J$ be an injective envelope of the additive generator $M$. Then, we have a surjection $\Ext^1_\L(J,M) \twoheadrightarrow \Ext_\L^1(M,M)$, thus $E \subset \G_0eE \subset \G e \G$. \end{proof}

The following description of $\md\sG$ inside $\smd\G$ is crucial for the proof of Theorem \ref{ungraded}.
\begin{Prop}\label{motothm}
Under the identification $\md\sG \subset \smd \G$ in Proposition \ref{emb}, we have
\[ \md\sG=\{ X \in \smd \G \mid \Ext_\G^i(\G_0,X)=0 , \ i=1,2 \}. \]
\end{Prop}

We need several lemmas for the proof. The first one shows the one of the inclusions stated in the proposition.
\begin{Lem}\label{subset}
For any $X \in \md\sG$, we have $\Ext^i_\G(\G_0,X)=0$ for $i=1,2$.
\end{Lem}
\begin{proof}For $i=1, 2$, we have $\Ext_\G^i(\G_0,X)=\sHom_\G(\G_0,X[i])=\bigoplus_{j \in \Z}\sHom^\Z_\G(\G_0,X[i+3j])=\bigoplus_{j \in \Z}\Hom_{\rD(\sG)}(\sG,X[i+3j])$, which is zero by $\gd\sG\leq2$. \end{proof}

\begin{Lem}\label{gammae}
\begin{enumerate}
\item\label{rigidrigid} $\Ext_\G^1(\G_0,\G_0)=0$ and $\Ext^2_\G(\G_0,\G_0)=0$.
\item\label{gammaegammae} $\sHom_\G(\G_0,E)=0$ and $\Ext^2_\G(\G_0,E)=0$.
\end{enumerate}
\end{Lem}
\begin{proof} Note that by $[3]={\mathrm{id}}$ and $\Om \G_0=E$, these statements are equivalent, so we prove the first one. By Proposition \ref{subset}, it is enough to show $\G_0 \in \md\sG$. But this is clear since $\G_0$ corresponds to $\sG$ under the equivalence $\rD(\md\sG) \simeq \smd^\Z\!\G$.
\end{proof}

Now we are in the position to prove Proposition \ref{motothm}.
\begin{proof}[Proof of Proposition \ref{motothm} ] The inclusion `$\subset$' is proved in Lemma \ref{subset}. We will show the converse inclusion. Let $X \in \smd \G$ be an object with $\Ext^i_\G(\G_0,X)=0$ for $i=1,2$. Consider the exact sequence $0 \to E \to \G \to \G_0 \to 0$ of $\G$-modules. Applying $\Hom_\G(-,X)$, we have an exact sequence
\begin{equation}\label{eq5}
\xymatrix{ 0 \ar[r] & Y \ar[r] & X \ar[r] & Z \ar[r] & \Ext_\G^1(\G_0,X)=0 }
\end{equation}
in $\md \G$, where $Y=\Hom_\G(\G_0,X)$ is the submodule of $X$ annihilated by $E$, and $Z=\Hom_\G(E,X)$. Note that $Y \in \md\G_0$ and $Z \in \md\sG$. Indeed, $Y \in \md\G_0$ holds since $Y$ is annihilated by $E$, and also $Z \in \md\sG$ holds by Lemma \ref{ring} since $eZ=\Hom_\G(Ee,X)=0$.

To prove that $X \in \md\sG$, we may assume $Y$ has no $\G$-projective (=injective) summands, since such summands are also summands of $X$.

Applying the functor $\Hom_\G(\G_0,-)$ to the sequence (\ref{eq5}) yields an exact sequence $\Ext^1_\G(\G_0,Z) \to \Ext_\G^2(\G_0,Y) \to \Ext_\G^2(\G_0,X)$. We have $\Ext^1(\G_0,Z)=0$ by Lemma \ref{subset} and $\Ext_\G^2(\G_0,X)=0$ by assumption, and hence $\Ext_\G^2(\G_0,Y)=0$.

Now, let $0 \to P_2 \to P_1 \to P_0 \to Y \to 0$ be a projective resolution of the $\G_0$-module $Y$ and consider it as an exact sequence in $\md^\Z\!\G$ (concentrated in degree 0). For $i=0, 1, 2,$ let $Q_i$ be the object corresponding to $P_i$ under $\smd^\Z\!\G \simeq \rD(\md\sG)$. Note that $Q_i \in \add\sG$. Then, the complex in $\rD(\md\sG)$ corresponding to $Y \in \smd^\Z\!\G$ is $Q=(\cdots \to 0 \to Q_2 \to Q_1 \to Q_0 \to 0 \to \cdots)$ with $Q_0$ at degree 0. Indeed, letting $W=\Ker(P_0 \to Y)$, it is the mapping cone of $P_2 \to P_1$ in $\smd^\Z\!\G$, so the corresponding object in $\rD(\md\sG)$ is the mapping cone of $Q_2 \to Q_1$, which is the complex $Q'=(\cdots \to 0 \to 0 \to Q_2 \to Q_1 \to 0 \to \cdots)$ with $Q_1$ at degree $0$. Similarly, $Y$ is the mapping cone of $W \to P_0$ in $\smd^\Z\!\G$, so the corresponding object in $\rD(\md\sG)$ is the mapping cone of $Q' \to Q_0$, which is precisely the complex $Q$.

We have 
\[ \sHom_\G(\G_0,Y)=\bigoplus_{i \in \Z}\sHom^\Z_\G(\G_0,Y(i))=\bigoplus_{i \in \Z}\Hom_{\rD(\sG)}(\sG,Q[3i])=\bigoplus_{i \in \Z}H^{3i}(Q)=H^0(Q), \]
and similarly, $\Ext_\G^1(\G_0,Y)=H^{-2}(Q)$ and $\Ext^2_\G(\G_0,Y)=H^{-1}(Q)$. Now, $\Ext_\G^2(\G_0,Y)=0$ shows that $Q$ is acyclic at degree $-1$. Then, since $\gd\sG \leq 2$, we see that $Q=H^0(Q) \oplus H^{-2}(Q)[2]$. Note that $H^{-2}(Q) \in \add Q_2$ since $Q_2 \to Q_1$ is a split epimorphism to its image. Therefore, $Y=Y' \oplus F$ in $\smd \G$ with $Y'=H^0(Q) \in \md\sG$ and $F= H^{-2}(Q)[2]\in \add\G_0[2]=\add E$. Note that $Y=Y' \oplus F$ also in $\md \G$, since we assumed $Y$ has no projective summands.

Consider the following diagram in $\md \G$ obtained by the push-out of (\ref{eq5}) along $Y \twoheadrightarrow F$.
\[ \xymatrix{
   0 \ar[r] & Y \ar[r]\ar@{>>}[d] & X \ar[r]\ar[d] & Z \ar[r]\ar@{=}[d] & 0 \\
   0 \ar[r] & F \ar[r] & U \ar[r] & Z \ar[r] & 0 } \]

We claim that $U$ is projective. Since $Y'=\Ker(Y \to F)=\Ker(X \to U)$, this claim will yield $X=Y' \in \md\sG$ in $\smd \G$, hence the theorem. To show that $U$ is projective, we show that $\sHom_\G(-,U)=0$. For this, it is sufficient to show that $\Ext_\G^i(\G_0,U)=0$ for $i=1,2,3$, since $\smd \G$ is generated by $\G_0$.

Applying $\sHom_\G(\G_0,-)$ to the above diagram yields
\[ \xymatrix@C=5mm{
   \underline{{}_\G(\G_0,Y)} \ar[r]\ar[d] & \underline{{}_\G(\G_0,X)} \ar[r]^a\ar[d] & \underline{{}_\G(\G_0,Z)} \ar[r]^{b \ \ }\ar@{=}[d] & \Ext^1_\G(\G_0,Y) \ar[r]\ar[d]^c & \Ext^1_\G(\G_0,X) \ar[d]\ar[r] & \Ext_\G^1(\G_0,Z) \ar@{=}[d] \\
   \underline{{}_\G(\G_0,F)} \ar[r] & \underline{{}_\G(\G_0,U)} \ar[r] & \underline{{}_\G(\G_0,Z)} \ar[r]^{d \ \ } & \Ext^1_\G(\G_0,F) \ar[r] & \Ext^1_\G(\G_0,U) \ar[r] & \Ext^1_\G(\G_0,Z), } \]
where $\underline{{}_\G(-,-)}$ stands for $\sHom_\G(-,-)$.
Now, the map $a$ is zero since $Y=\Hom_\G(\G_0,X)$, and hence $b$ is an isomorphism by $\Ext_\G^1(\G_0,X)=0$. Note also that $c$ is an isomorphism since $Y=Y' \oplus F$ and $\Ext_\G^1(\G_0,Y')=0$. Therefore, $d$ is also an isomorphism. This shows that $\sHom_\G(\G_0,U)=0$ and $\Ext_\G^1(\G_0,U)=0$, since $\sHom_\G(\G_0,F)=0$ (by $\sHom_\G(\G_0,E)=0$; Lemma \ref{gammae}(\ref{gammaegammae})) and $\Ext_\G^1(\G_0,Z)=0$ (by $Z \in \md\sG$ and Lemma \ref{subset}).
$\Ext_\G^2(\G_0,U)=0$ follows from $\Ext_\G^2(\G_0,F)=0$ (by Lemma \ref{gammae}(\ref{gammaegammae})) and $\Ext_\G^2(\G_0,Z)=0$ ($Z \in \md\sG$ and Lemma \ref{subset}). This completes the proof.
\end{proof}

Now the main result of this section is a consequence of Proposition \ref{motothm}.
\begin{proof}[Proof of Theorem \ref{ungraded}](\ref{5.10.1})  The inclusion `$\subset$' is done in Lemma \ref{gammae}(\ref{rigidrigid}). 
We show the inclusion `$\supset$'. Let $X \in \smd \G$ be in the right-hand-side. We see that $X \in \md\sG$ by Proposition \ref{motothm}. Then, we have $\Ext_\G^i(X,\G_0)=\bigoplus_{j \in \Z}\sHom^\Z_\G(X,\G_0[i](j))=\bigoplus_{j \in \Z}\Hom_{\rD(\sG)}(X,\sG[i+3j])$. (Here, $X$ denotes the identified objects through the functors in Proposition \ref{emb}.) The vanishing of this for $i=1,2$ shows that $\Ext^i_{\sG}(X,\sG)=0$ for $i=1,2$. Therefore we conclude that $X \in \add\sG$ by $\gd\sG \leq 2$.\\
(\ref{5.10.2})  This is immediate by (\ref{5.10.1}).\\
(\ref{ct})  Assume there exists a 3-cluster tilting subcategory in $\smd \G$ containing $\G_0$. Then, its preimage $\C$ under the functor $\rD(\sG) \simeq \smd^\Z\!\G \to \smd \G$ is a 3-cluster tilting subcategory of $\rD(\sG)$ containing $\sG$ and stable under $[3]$. On the other hand, the $\nu_3$-orbit 
\[ \mathcal{U}=\add\{\nu_3^i\sG \mid i \in \Z \} \subset \rD(\sG) \]
of $\sG$ is a 3-cluster tilting subcategory of $\rD(\sG)$ \cite[1.23]{Iy}, where $\nu=D\sG\otimes_{\sG}^L-$ is the Serre functor for $\rD(\sG)$ and $\nu_3=\nu \circ [-3]$. Then, since $\C$ is $\nu_3$-stable, we have $\mathcal{U} \subset \C$. Moreover, since $\mathcal{U}$ and $\C$ are cluster tilting subcategories, the equality $\mathcal{U}=\C$ must hold, and in particular, $\mathcal{U}$ is $[3]$-stable. Then by \cite[3.1]{IO}, $\sG$ is 3-representation-finite. Since $\gd\sG \leq 2$, $\sG$ has to be semisimple, or equivalently, $\smd \G$ is semisimple.
\end{proof}

\section{Examples}\label{examples}
We give some examples to illustrate our results.
We use the following notation for graded modules: the number of the symbols $(-)'$ (resp. $'(-)$) over the composition factor indicates that the factor lies in the positive (resp. negative) degree of that value.
 
The first one is a hereditary algebra, and we explain the results in Section \ref{hereditary}.
\begin{Ex}\label{A3}
\newcommand{\nagni}[2]{\begin{smallmatrix}#1 \\ #2 \end{smallmatrix}}
\newcommand{\simp}[1]{\begin{smallmatrix}#1\end{smallmatrix}}
\newcommand{\masan}[3]{\begin{smallmatrix}{} #1 \\ #2 \\ #3 \end{smallmatrix}}
\newcommand{\bigmasan}[3]{\begin{array}{l} #1 \\ #2 \\ #3 \end{array}}
\newcommand{\nagsan}[3]{\begin{smallmatrix}#1 & & #2 \\ & #3 & \end{smallmatrix}}
\newcommand{\nagson}[3]{\begin{smallmatrix}& #1 & \\ #2 & & #3 \end{smallmatrix}}
\newcommand{\yon}[4]{\begin{smallmatrix}& #1 & \\ #2 & & #3 \\ &#4& \end{smallmatrix}}
\newcommand{\bigyon}[4]{\begin{array}{l@{\hspace{5pt}}l@{\hspace{5pt}}l}& #1 & \\ #2 & & #3 \\ &#4& \end{array}}
\newcommand{\gra}[1]{\begin{smallmatrix} (#1) \end{smallmatrix}}
\def\hh{\hspace{2pt}}
\def\gg{\hspace{-2pt}}
Let $\L$ be the path algebra of linearly oriented type $A_3$:
\[ \xymatrix{ \circ \ar[r] & \circ \ar[r] & \circ }. \]
It is a representation-finite hereditary algebra and its derived category is presented by the quiver $\Z A_3$ with mesh relations:
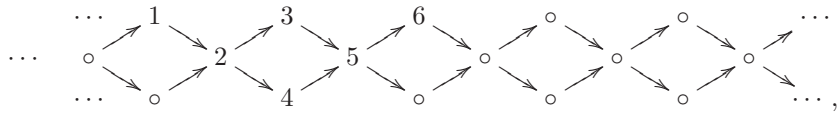
\[ {
   \xymatrix@!C@!R@C=1.2pt@R=1.2pt{ & \cdots & 1 \ar[dr] & & 3 \ar[dr] & & 6 \ar[dr] & & \circ \ar[dr] & & \circ\ar[dr] & & \cdots \\
                            \cdots & \circ\ar[dr]\ar[ur] & & 2 \ar[dr]\ar[ur] & & 5 \ar[dr]\ar[ur] & & \circ \ar[dr]\ar[ur] & & \circ\ar[dr]\ar[ur] & & \circ\ar[dr]\ar[ur] \\
& \cdots & \circ\ar[ur] & & 4 \ar[ur] & & \circ \ar[ur] & & \circ \ar[ur] & & \circ\ar[ur] & & \cdots, } 
}\]
where $1, \ldots, 6$ are the objects from $\md\L$. Then by Proposition \ref{waru1}, the Yoneda algebra $\G$ of $\L$ is presented by the quiver $\Z A_3/[1]$:
\[ {
   \xymatrix@!C@!R@C=1.2pt@R=1.2pt{ & \cdots & 1 \ar[dr] & & 3 \ar[dr] & & 6 \ar[dr] & & 4 \ar[dr] & & 1 \ar[dr] & & \cdots \\
                            \cdots & 5\ar[dr]\ar[ur] & & 2 \ar[dr]\ar[ur] & & 5 \ar[dr]\ar[ur] & & 2\ar[dr]\ar[ur] & & 5\ar[dr]\ar[ur] & & 2\ar[dr]\ar[ur] \\
& \cdots & 6\ar[ur] & & 4 \ar[ur] & & 1\ar[ur] & & 3\ar[ur] & & 6\ar[ur] & & \cdots, } 
}\]
with mesh relations, and the arrows from $5 \to 1$ and $6 \to 2$ have degree $1$, while the others have degree $0$. 
The composition series of the projectives look
\[ \G=\bigmasan{1}{5'}{3'} \oplus \bigyon{2}{1}{{\gg}6'}{5'} \oplus \bigmasan{3}{2}{6'} \oplus \bigmasan{4}{2}{1} \oplus \bigyon{5}{3}{4}{2} \oplus \bigmasan{6}{5}{4}, \]
We see $\G$ is indeed self-injective, as stated in Proposition \ref{iffhered}. The AR-quiver of $\md^\Z\!\G$ is computed to be
\[ \xymatrix@!C=0.6mm@!R=0.6mm{
\cdots\ar[dr] && {\simp{6}\gra{-1}} \ar[dr] & &{\nagni{2}{1}}\ar[dr]\ar[r] &{\masan{4}{2}{1}}\ar[r]& {\nagni{4}{2}}\ar[dr] & & {\simp{3}}\ar[dr] & & {\nagni{5}{4}}\ar[dr]\ar[r] &{\masan{6}{5}{4}}\ar[r]& {\nagni{6}{5}}\ar[dr] && \cdots \\
                                   & {\nagsan{1}{{\gg}6'}{5'}} \ar[dr]\ar[ur]\ar[r] &{\yon{2{\hh}}{1}{{\gg}6'}{5'}}\ar[r]& {\nagson{2}{1}{6'}} \ar[dr]\ar[ur] & & {\simp{2}}\ar[dr]\ar[ur] & & {\nagsan{3}{4}{2}}\ar[dr]\ar[ur]\ar[r] &{\yon{5}{3}{4}{2}}\ar[r]& {\nagson{5}{3}{4}} \ar[ur]\ar[dr] & & {\simp{5}}\ar[dr]\ar[ur] & & {\nagsan{1}{{\gg}6'}{5'}\gra{1}}\ar[dr]\ar[ur]\ar[r] & \\
                                   \cdots\ar[ur] && {\simp{1}}\ar[ur] & & {\nagni{2^{\hh}}{6'}} \ar[ur]\ar[r] &{\masan{3^{\hh}}{2^{\hh}}{6'}}\ar[r]& {\nagni{3}{2}}\ar[ur] & & {\simp{4}} \ar[ur] & & {\nagni{5}{3}}\ar[ur]\ar[r] &{\masan{1^{\hh}}{5'}{3'}\gra{1}}\ar[r]& {\nagni{'1}{{}^{\hh}5}}\ar[ur] && \cdots.} \]
We deduce a triangle equivalence $\sCM^\Z\!\G \simeq \rD(A_3)$, which certifies Theorem \ref{tilting}.

We next consider the ungraded module category $\smd\G$, whose AR-quiver is the following:
\definecolor{gr}{gray}{0.85}
\setlength{\fboxrule}{0.8pt}
\newcommand{\cobox}[1]{\fcolorbox{black}{gr}{#1}}
\[ \xymatrix@!C=0.6mm@!R=0.6mm{  & {\simp{6}} \ar[dr] & &\cobox{$\nagni{2}{1}$}\ar[dr] & & {\nagni{4}{2}}\ar[dr] & & \colorbox{gr}{$\simp{3}$}\ar[dr] & & {\nagni{5}{4}}\ar[dr] & & {\nagni{6}{5}} \ar[dr] & \ar@{-->}[dd] \\
                                   {\nagsan{1}{6}{5}} \ar[dr]\ar[ur] & & {\nagson{2}{1}{6}} \ar[dr]\ar[ur] & & \colorbox{gr}{$\simp{2}$}\ar[dr]\ar[ur] & & {\nagsan{3}{4}{2}}\ar[dr]\ar[ur] & & {\nagson{5}{3}{4}} \ar[ur]\ar[dr] & & {\simp{5}}\ar[dr]\ar[ur] & & {\nagsan{1}{6}{5}} \\
                                   \ar@{-->}[uu] & \cobox{$\simp{1}$}\ar[ur] & & {\nagni{2}{6}} \ar[ur] & & \cobox{$\nagni{3}{2}$}\ar[ur] & & {\simp{4}} \ar[ur] & & {\nagni{5}{3}}\ar[ur] & & {\nagni{1}{5}} \ar[ur] & \quad ,} \]
where two ends are identified along the dotted lines. Then, $\G_0 \in \smd\G$ has $3$ indecomposable summands, which are in the boxes.

On the one hand, $\md\sG$ consists of $\G$-modules which are annihilated by the idempotents $e_4, e_5, e_6$ of $\G$ corresponding to injective $\L$-modules by Lemma \ref{ring}. Therefore, they are precisely $\G$-modules having only 1, 2, or 3 as composition factors, hence are the shaded modules.
On the other hand, a computation of the subcategory $\{ X \in \smd \G \mid \Ext_\G^{1,2}(\G_0,X)=0 \}$ shows that it actually coincides with $\md\sG$.
This demonstrates Proposition \ref{motothm}.

If there is a 3-cluster tilting subcategory $\U$ containing $\G_0$, since $\U$ is $\nu_3$-stable, we see that it has to be the whole $\smd \G$, which is absurd. This shows Theorem \ref{ungraded}(\ref{ct}) in this case.
\end{Ex}

The next two examples are algebras of finite global dimension and we demonstrate the results in Section \ref{finitegldim}.
\begin{Ex}\label{kantan}
\newcommand{\bigni}[2]{\begin{array}{l}#1 \\ #2 \end{array}}
\newcommand{\bigsan}[3]{\begin{array}{l} #1 \\ #2 \\ #3 \end{array}}
\newcommand{\bigsam}[3]{\begin{array}{r} #1 \\ #2 \\ #3 \end{array}}
\newcommand{\bigyon}[4]{\begin{array}{l@{\hspace{5pt}}l@{\hspace{5pt}}l}& #1 & \\ #2 & & #3 \\ &#4& \end{array}}
\def\hh{\hspace{2pt}}
\def\gg{\hspace{-2pt}}
Let $\L$ and $\G$ be as in the previous example.
Let $\L'$ be an algebra presented by the following quiver with relations:
\[ \xymatrix{ \circ \ar[r]^a & \circ \ar[r]^b & \circ }, \quad ab=0. \]
It has global dimension $2$ and is of finite representation type. The category $\md\L'$ is located in $1, 4, 5, 6, 3[1]$ in $\rD(\L)$ from Example \ref{A3}. Therefore, the Yoneda algebra $\G'$ of $\L'$ is isomorphic to $e\G e$ for the idempotent $e \in \G$ corresponding to these five summands.
Renumbering the vertices, $\G'$ is presented by the following quiver
\[ {
   \xymatrix@!C@!R@C=1.2pt@R=1.2pt{ & \cdots & 1 \ar[ddrr] & & 5 \ar[dr] & & 4 \ar[ddrr] & & 2 \ar[dr] & & 1 \ar[ddrr] & & \cdots \\
                            \cdots & \ar[dr]\ar[ur] & & & & 3 {} \ar[dr]\ar[ur] & & & & 3 \ar[dr]\ar[ur] & & \\
& \cdots & 4 \ar[uurr] & & 2 \ar[ur] & & 1 \ar[uurr] & & 5 \ar[ur] & & 4 \ar[uurr] & & \cdots, } 
}\]
with relations induced from the mesh relations on $\Z A_3$, and the arrows $3 \to 1$ and $5 \to 3$ have degree $1$, and others have degree $0$. The projective and injective $\G'$-modules look
\[
\begin{array}{r@{=}c@{\oplus}c@{\oplus}c@{\oplus}c@{\oplus}c}
\G'&\bigsan{1}{3'}{5''} & \bigni{2}{1} & \bigyon{3}{2}{5'}{} & \bigsan{4}{3}{2} & \bigni{5}{4}, \\
D\G'&\bigni{2}{1} & \bigsan{4}{3}{2} & \bigyon{}{{}'1}{4}{3} & \bigni{5}{4} & \bigsam{''1}{'3}{5}. 
\end{array}
\]
It is easily checked that $\G'$ is indeed a Gorenstein algebra of dimension $1$, as in Theorem \ref{YonIG}. The AR-quiver of $\md^\Z\!\G'$ is computed to be
\newcommand{\ichi}[1]{\begin{smallmatrix}#1\end{smallmatrix}}
\newcommand{\nagni}[2]{\begin{smallmatrix}#1 \\ #2 \end{smallmatrix}}
\newcommand{\san}[3]{\begin{smallmatrix}#1 \\ #2 \\ #3 \end{smallmatrix}}
\newcommand{\niichi}[3]{\begin{smallmatrix}#1 & & #2 \\ & #3 & \end{smallmatrix}}
\newcommand{\ichini}[3]{\begin{smallmatrix}& #1 & \\ #2 & & #3 \end{smallmatrix}}
\newcommand{\gra}[1]{\begin{smallmatrix} (#1) \end{smallmatrix}}
\[ \xymatrix@!R=0.5mm@!C=0.5mm{
    && {\san{1^{\hh\hh}}{3'^{\hh}}{5''}} \ar[dr] &&&& {\nagni{5}{4}\gra{-1}} \ar[dr] &&&& \\
    & *+[Fo:<5mm>]{\nagni{3^{\hh}}{5'}\gra{-1}} \ar[dr]\ar[ur] && {\nagni{1^{\hh}}{3'}} \ar[dr] && *+[Fo:<5mm>]{\ichi{4}\gra{-1}} \ar[ur] && *+[Fo:<5mm>]{\ichi{5}\gra{-1}} \ar[dr] && *+[Fo:<4mm>]{\nagni{3}{2}} \ar[dr] & \\
    \cdots \ar[ur]\ar[dr] && {\ichi{3}\gra{-1}} \ar[dr]\ar[ur] && {\niichi{1}{{\gg}4'}{3'}} \ar[dr]\ar[ur] &&&& {\ichini{3}{2}{5'}} \ar[dr]\ar[ur] && \cdots \\
    & *+[Fo:<5mm>]{\nagni{3}{2}\gra{-1}} \ar[dr]\ar[ur] && {\nagni{4}{3}\gra{-1}} \ar[ur] && *+[Fo:<3mm>]{\ichi{1}} \ar[dr] && *+[Fo:<3mm>]{\ichi{2}} \ar[ur] && *+[Fo:<4mm>]{\nagni{3^{\hh}}{5'}} \ar[ur] & \\
    && {\san{4}{3}{2}\gra{-1}} \ar[ur] &&&& {\nagni{2}{1}} \ar[ur] &&&& \quad.} \]
The Cohen-Macaulay $\G'$-modules are the circled modules and the projective modules.
Consequently, we have $\sCM^\Z\!\G' \simeq \rD(A_1 \times A_1)$, as was proved in Theorem \ref{tilting}.

As in Section \ref{3.5}, we now consider the second Veronese subalgebra $\G'^{(2)}=\Hom_\L(M,M) \oplus \Ext_\L^2(M,M)$ of $\G'$. We see as above that it is presented by the following quiver with relations:
\[ \xymatrix@C=5mm{
   \cdots \ar[r] & 5 \ar[r]^e & 1 \ar[r]^a & 2 \ar[r]^b & 3 \ar[r]^c & 4 \ar[r]^d & 5 \ar[r]^e & 1 \ar[r] & \cdots }, \quad ea=0, \ ab=0, \ cd=0, \ de=0. \]
It is easily verified that this algebra has global dimension $5$, which illustrates Theorem \ref{omakeyon}.
\end{Ex}

\begin{Ex}\label{fukuzatsu}
Let $\L$ be an algebra presented by the following quiver with relations:
\[ \xymatrix{ \circ \ar@<2pt>[r]^a & \circ \ar@<2pt>[l]^b}, \quad ab=0. \]
It has global dimension 2 and is of finite representation type. Its Yoneda algebra $\G$ is presented by the following graded quiver with relations:
\[ \xymatrix@R=4mm{
    &&& 5 \ar[dr]^d &&& \\
    \cdots \ar@<2pt>[dr]^f\ar@<-2pt>[dr]_y && 3 \ar[ur]^b\ar[dr]^c && 4 \ar@<2pt>[dr]^f\ar@<-2pt>[dr]_y && \cdots \\
    & 1 \ar@<2pt>[ur]^a\ar@<-2pt>[ur]_x && 2 \ar[ur]^e && 1 \ar@<2pt>[ur]^a\ar@<-2pt>[ur]_x &, } \]
\[ \deg a= \cdots=\deg f=0, \ \deg x=\deg y=1, \] \[ ac=0, \ bd=ce, \ ef=0, \ fx=ya, \ xb=0, \ dy=0. \] 
The composition series of the projective and the injective modules look
\def\hhh{\hspace{5pt}}
\def\hhhh{\hspace{3pt}}
\def\pr{\begin{array}{l@{\hhh}l@{\hhh}l}&1& \\ 4 && 4' \\ 5 && 2'\end{array}}
\def\ppr{\begin{array}{l} 2 \\ 3 \\ 1' \\ 4'' \\ 2'' \end{array}}
\def\pppr{\begin{array}{l@{\hhh}l@{\hhh}l@{\hhh}l@{\hhhh}l} &&3&& \\ &1&&1'& \\ 4&&4'&&4'' \\ 5&&&&2'' \end{array}}
\def\ppppr{\begin{array}{c@{\hhh}c@{\hhh}c} &4& \\ 5&&2 \\ &3& \end{array}}
\def\pppppr{\begin{array}{c} 5 \\ 3 \\ 1 \\ 4 \\ 5 \end{array}}
\def\ij{\begin{array}{r@{\hhh}r@{\hhh}r} 5&&{}'2 \\ 3&&'3 \\ &1& \end{array}}
\def\iij{\begin{array}{r} ''2 \\ {}''3 \\ {}'1 \\ 4 \\ 2 \end{array}}
\def\iiij{\begin{array}{c@{\hhh}c@{\hhh}c} &4& \\ 5&&2 \\ &3& \end{array}}
\def\iiiij{\begin{array}{r@{\hhh}r@{\hhh}r@{\hhh}r@{\hhhh}r} 5&&&&''2 \\ 3&&'3&&''3 \\ &1&&'1& \\ &&4&& \end{array}}
\def\iiiiij{\begin{array}{c} 5 \\ 3 \\ 1 \\ 4 \\ 5 \end{array}}
\[
\begin{array}{r@{{\hspace{2pt}}=}c@{\oplus}c@{\oplus}c@{\oplus}c@{\oplus}l}
\G & \pr & \ppr & \pppr & \ppppr & \pppppr \\
\rule{0pt}{43pt}
D\G &\ij & \iij & \iiij & \iiiij & \iiiiij. 
\end{array}
\]
Using this we can check that $\G$ is indeed 1-Gorenstein (Theorem \ref{YonIG}).
Moreover, the AR-quiver of $\sCM^\Z\G$ is as follows:
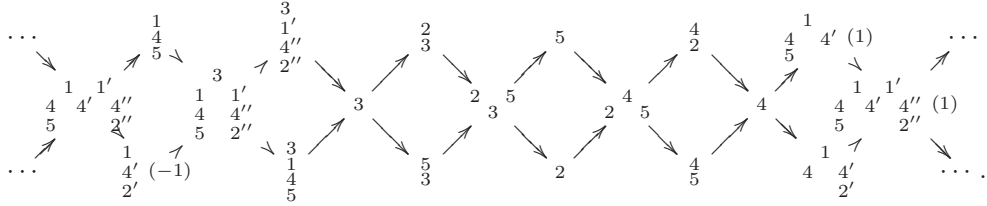
\def\hh{\hspace{2pt}}
\def\gg{\hspace{-2pt}}
\newcommand{\ichi}[1]{\begin{smallmatrix}#1\end{smallmatrix}}
\newcommand{\nagni}[2]{\begin{smallmatrix}#1 \\ #2 \end{smallmatrix}}
\newcommand{\san}[3]{\begin{smallmatrix}#1 \\ #2 \\ #3 \end{smallmatrix}}
\def\niichi{\begin{smallmatrix}2 & & 5 \\ & 3 & \end{smallmatrix}}
\def\ichini{\begin{smallmatrix}& 4 & \\ 2 & & 5 \end{smallmatrix}}
\newcommand{\yon}[4]{\begin{smallmatrix}#1 \\ #2 \\ #3 \\ #4 \end{smallmatrix}}
\def\migi{\begin{smallmatrix} & 1 & \\ 4&&4' \\ &&2' \end{smallmatrix}}
\def\hidari{\begin{smallmatrix} &1& \\ 4&&4' \\ 5&& \end{smallmatrix}}
\def\dai{\begin{smallmatrix} &3& \\ 1&&1'^{\hh} \\ 4&&4'' \\ 5&&2'' \end{smallmatrix}}
\def\saidai{\begin{smallmatrix} &1&&{\gg}1'& \\ 4&&{\gg}4'&& {\gg\gg}4'' \\ 5&&&&{\gg\gg}2'' \end{smallmatrix}}
\newcommand{\gra}[1]{\begin{smallmatrix} (#1) \end{smallmatrix}}
\[ \xymatrix@!R=0.5mm@!C=0.5mm{
   \cdots\ar[dr] && {\san{1}{4}{5}}\ar[dr] && {\yon{3^{\hh\hh}}{1'^{\hh}}{4''}{2''}}\ar[dr] && {\nagni{2}{3}}\ar[dr] && {\ichi{5}}\ar[dr] && {\nagni{4}{2}}\ar[dr] && {\hidari\gra{1}}\ar[dr] && \cdots \\
   & {\saidai}\ar[dr]\ar[ur] && {\dai}\ar[dr]\ar[ur] && {\ichi{3}}\ar[dr]\ar[ur] && {\niichi}\ar[dr]\ar[ur] && {\ichini}\ar[dr]\ar[ur] && {\ichi{4}}\ar[dr]\ar[ur] && {\saidai\gra{1}}\ar[dr]\ar[ur] & \\
   \cdots \ar[ur]&& {\san{1^{\hh}}{4'}{2'}\gra{-1}}\ar[ur]&& {\yon{3}{1}{4}{5}}\ar[ur]&& {\nagni{5}{3}}\ar[ur]&& {\ichi{2}}\ar[ur]&& {\nagni{4}{5}}\ar[ur]&& {\migi}\ar[ur]&& \cdots. } \]
We can verify the triangle equivalence $\sCM^\Z\!\G \simeq \rD(A_3)$ (Theorem \ref{tilting}).

Now let us consider the second Veronese subalgebra $\G^{(2)}=\Hom_\L(M,M) \oplus \Ext_\L^2(M,M)$ of $\G$. It is presented by
\[ \xymatrix@R=4mm{
    &&& 5 \ar[dr]^d &&& \\
    \cdots \ar[dr]_f\ar[rr]^z && 3 \ar[ur]^b\ar[dr]^c && 4 \ar[dr]_f\ar[rr]^z && \cdots \\
    & 1 \ar[ur]_a && 2 \ar[ur]^e && 1 \ar[ur]_a &, } \quad
    \xymatrix@R=2mm{ \\
    ac=0, \ bd=ce, \ ef=0, \\
    zb=0, \ dz=0. }\]
We can easily check that $\G^{(2)}$ has global dimension $5$ as in Theorem \ref{omakeyon}.
\end{Ex}

The next examples are self-injective algebras of infinite global dimension. We illustrate the results in Section \ref{infinitegldim}.
\begin{Ex}
Let $\L$ be a self-injective Nakayama algebra with $n$ vertices and Loewy length $l+1$, so it is presented with the following cyclic quiver with relations:
\[ \xymatrix@C=6mm{ \cdots\ar[r]^-a& n-1 \ar[r]^-a& 0\ar[r]^-a&1\ar[r]^-a&2\ar[r]^-a&\cdots\ar[r]^-a&n-1\ar[r]^-a&0\ar[r]^-a&\cdots }, \quad a^{l+1}=0. \]
Its Yoneda algebra $\G$ is $1$-Gorenstein by Theorem \ref{2IG}. 
By Theorem \ref{tilting2}, we have a triangle equivalence $\rD(\md\sG) \simeq \sGP^\Z\!\G$, with the stable Auslander algebra $\sG$ of $\L$ presented by the following quiver with mesh relations:
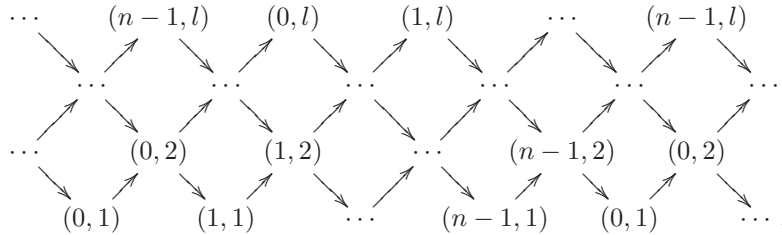
\[ \xymatrix@!R=0.5mm@!C=0.5mm{
    \cdots\ar[dr]&& (n-1,l)\ar[dr]&& (0,l)\ar[dr]&& (1,l)\ar[dr]&& \cdots\ar[dr]&& (n-1,l) \ar[dr]\\
    &\cdots\ar[dr]\ar[ur]&&\cdots\ar[dr]\ar[ur]&&\cdots\ar[dr]\ar[ur]&&\cdots\ar[dr]\ar[ur]&&\cdots\ar[dr]\ar[ur]&& \cdots \\
    \cdots\ar[dr]\ar[ur]&& (0,2)\ar[dr]\ar[ur]&& (1,2)\ar[dr]\ar[ur]&& \cdots\ar[dr]\ar[ur]&& (n-1,2)\ar[dr]\ar[ur]&& (0,2)\ar[dr]\ar[ur] \\
    &(0,1)\ar[ur]&& (1,1)\ar[ur]&& \cdots\ar[ur]&& (n-1,1)\ar[ur]&& (0,1)\ar[ur]&& \cdots.} \]
\end{Ex}

\begin{Ex}
We specialize the above example to $n=1$; Let $\L=k[x]/(x^{l+1})$. Then the Yoneda algebra $\G$ of $\L$ is presented by the following graded quiver with relations:
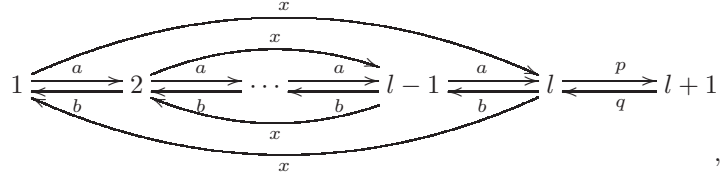
\[ \xymatrix@R=5mm@C=12mm{ \\
    1 \ar@<2pt>[r]^a\ar@/^24pt/[rrrr]<2pt>^x& 2 \ar@<2pt>[l]^b\ar@<2pt>[r]^a\ar@/^12pt/[rr]<2pt>^x& \cdots\ar@<2pt>[l]^b\ar@<2pt>[r]^a& l-1\ar@<2pt>[l]^b\ar@<2pt>[r]^a\ar@/^12pt/[ll]<2pt>^x& l\ar@<2pt>[l]^b\ar@<2pt>[r]^p\ar@/^24pt/[llll]<2pt>^x & l+1\ar@<2pt>[l]^q \\
    &&&&& \qquad, } \]
\[ \deg a= \deg b= \deg p =\deg q =0, \ \deg x=1, \]
\[ ab=ba, \ ba=pq, \quad \ ax=xb, \ bx=xa, \quad xp=0, \ qx=0. \]
By Theorem \ref{2IG}, this is a Gorenstein algebra of dimension $1$. Also, the stable Auslander algebra of $\L$ is the preprojective algebra $\Pi$ of type $A_l$. Therefore by Theorem \ref{tilting2}, we have a triangle equivalence $\rD(\md\Pi) \simeq \sCM^\Z\!\G$.

\renewcommand{\arraystretch}{0.7}
\def\hhh{\hspace{5pt}}
\def\big{\begin{array}{c} 1 \\ 1 \\ 1 \\ \cdot \\ \cdot \end{array}}
\def\chiisai{\begin{array}{c} 1 \\ 2 \end{array}}
We take a closer look for the case $l=1$; Let $\L=k[x]/(x^2)$. This is a self-injective algebra with two indecomposable modules, namely the simple module $S$ and the free module $\L$. Its Yoneda algebra $\G$ is $1$-Gorenstein by Theorem \ref{2IG}. It is presented by the following graded quiver with relations:
\[ \xymatrix{
    1  \ar@(ul,dl)[]_x \ar@<2pt>[r]^p & 2\ar@<2pt>[l]^q }, \quad \deg p=\deg q=0, \ \deg x=1, \quad pq=0, \ xp=0, \ qx=0. \] 
The composition series of of the projective $\G$-modules look
\[ \G=\begin{array}{l@{\hhh}l@{\hhh}l} &1&\\2&&1'\\&&1''\\&&\cdot\\&&\cdot \end{array} \oplus \begin{array}{c} 2\\1\\2 \end{array} \]
Now, by the triangle $S \to \L \to S \to S[1]$ in $\rD(\L)$, we get three non-projective Cohen-Macaulay $\G$-modules;
\[ \begin{array}{c} 2 \end{array}, \quad \big, \quad \chiisai. \]
By the triangle equivalence $\rD(\md k) \simeq \sGP^\Z\!\G$ of Theorem \ref{tilting2}, we see that these are all the graded Cohen-Macaulay $\G$-modules up to degree shift.
\end{Ex}

We end this paper by the following example of $\L$ which is not even Gorenstein.
\begin{Ex}
Let $\L$ be an algebra presented by the following quiver with relations:
\[ \xymatrix{
    \circ \ar@(ul,dl)[]_a \ar[r]^b & \circ }, \quad a^2=0, \ ab=0. \]
This is of finite representation type and its Yoneda algebra $\G$ is presented by the graded quiver
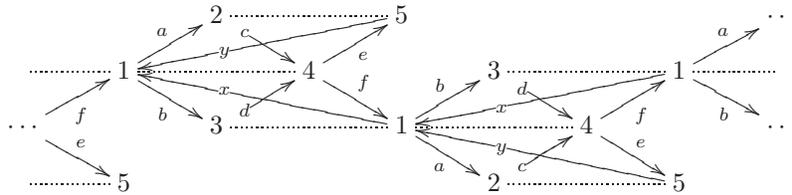
\[ \xymatrix@R=3mm{
    && 2\ar@{.}[rr]\ar[dr]|<<<c&& 5\ar[dlll]|>>>>>>>>>>>>y&&&& \cdots \\
    \ar@{.}[r]& 1\ar@{.}[rr]\ar[dr]_b\ar[ur]^a&& 4 \ar[ur]_e\ar[dr]^f&& 3\ar@{.}[rr]\ar[dr]|<<<d&& 1\ar@{.}[r]\ar[ur]^a\ar[dlll]|>>>>>>>>>>>>x\ar[dr]_b& \\
    \cdots\ar[ur]_f\ar[dr]^e&& 3\ar@{.}[rr]\ar[ur]|<<<d&& 1\ar@{.}[rr]\ar[ulll]|>>>>>>>>>>>>x\ar[ur]^b\ar[dr]_a&& 4\ar[ur]_f\ar[dr]^e&& \cdots \\
    \ar@{.}[r]& 5&&&& 2\ar@{.}[rr]\ar[ur]|<<<c&& 5\ar[ulll]|>>>>>>>>>>>>y & \quad,} \]
\[ \deg a= \cdots =\deg f=0, \quad \deg x= \deg y =1, \]
with relations given as follows:
\begin{itemize}
\item the mesh relations along the dotted lines.
\item a commutativity relation $fx=ey$.
\item any path of degree $1$ which strictly contains the AR sequences $1 \xrightarrow{x} 1 \xrightarrow{b} 3$, $(4 \xrightarrow{f} 1 \xrightarrow{x} 1)=(4 \xrightarrow{e} 5 \xrightarrow{y} 1)$, or $5 \xrightarrow{y} 1 \xrightarrow{a} 2$ is zero.
\item vanishing of extensions: $xa=0, \ yb=0$
\end{itemize}
The stable Auslander algebra of $\L$ is the path algebra of non-linearly oriented type $A_3$. Therefore, by Theorem \ref{tilting2}, we deduce a triangle equivalence $\sGP^\Z\!\G \simeq \rD(A_3)$.
\end{Ex} 

\begin{appendix}
\renewcommand{\theequation}{\Alph{section}.\arabic{equation}}
\section{Existence of realization functors}
In this appendix we give a proof of the existence of a realization functor for algebraic triangulated categories. This is well-known to experts, but we include a proof for the convenience of the reader. Similar constructions are done in \cite[Section 3]{BGSc}, \cite[Section 3]{ChR}.
In the first subsection we recall the notion of f-categories and in the second subsection we give a formulation of a filtered derived category of a DG category. \subsection{F-categories and realization functors}
Let $\T$ be a triangulated category with a $t$-structure, and let $\H$ be its heart. It is a classical problem in the theory of triangulated categories to compare $\T$ with the derived category $\rD(\H)$. However, it is not clear whether we can extend the inclusion $\H \subset \T$ to a triangle functor $\rD(\H) \to \T$ because of the lack of the functoriality of mapping cones in a triangulated category. It was shown in \cite{BBD} that a realization functor $\rD(\H) \to \T$ exists under the assumption that $\T$ has an `f-category' over itself.

We first recall the notion of f-categories (or filtered enhancements) of a given triangulated category.
\begin{Def}[{\cite[Appendix]{B}\cite[3.1]{PV}}]
Let $\T$ be a triangulated category. An {\it f-category} over $\T$ is a triangulated category $\X$ endowed with full triangulated subcategories $\X(\geq\!0)$, $\X(\leq\!0)$ such that $\T \simeq \X(\geq\!0)\cap\X(\leq\!0)$, an autoequivalence $s$ of $\X$, and a natural transformation $\a \colon 1_\X \to s$, satisfying the following conditions, where we put $\X(\geq\!n)=s^n\X(\geq\!0)$ and $\X(\leq\!n)=s^n\X(\leq\!0)$.
\begin{enumerate}
\item\label{one} $\Hom_\X(\X(\geq\!1), \X(\leq\!0))=0$.
\item\label{two} For any $X \in \X$, there exists a triangle
\[ \xymatrix{
   X_{\geq 1} \ar[r]^{a_X} & X \ar[r]^{b_X} & X_{\leq 0} \ar[r]^{c_X} &X_{\geq 1}[1] } \]
with $X_{\geq 1} \in \X(\geq\!1)$ and $X_{\leq 0} \in \X(\leq\!0)$.
\item\label{three} $\X=\bigcup_{n \in \Z}\X(\geq\!n)=\bigcup_{n \in \Z}\X(\leq\!n)$.
\item\label{four} $\X(\geq\!1) \subset \X(\geq\!0)$ and $\X(\leq\!0) \subset \X(\leq\!1)$.
\item\label{five} $\a_X=s(\a_{s^{-1}X})$ in $\Hom_\X(X,sX)$.
\item\label{six} $\Hom_\X(sX,Y) \xrightarrow{\a_X\cdot} \Hom_\X(X,Y)$ and $\Hom_\X(X,s^{-1}Y) \xrightarrow{\cdot \a_{s^{-1}Y}} \Hom_\X(X,Y)$ are isomorphisms for all $X \in \X(\leq\!0)$ and $Y \in \X(\geq\!1)$.
\item\label{nana} For any morphism $f \colon X \to Y$ in $\X$, the commutative diagram of triangles
\[ \xymatrix{
   X_{\geq 1} \ar[r]^{a_X}\ar[d]_{f_{\geq1}\cdot\a_{Y_{\geq1}}} & X \ar[r]^{b_X}\ar[d]_{f\cdot\a_Y} & X_{\leq 0} \ar[r]^{c_X}\ar[d]^{f_{\leq0}\cdot\a_{Y_{\leq0}}} & X_{\geq 1}[1]\ar[d]^{(f_{\geq1}\cdot\a_{Y_{\geq1}})[1]} \\
   s(Y_{\geq 1}) \ar[r]_{s(a_Y)} & sY \ar[r]_{s(b_Y)} &  s(Y_{\leq 0}) \ar[r]_{s(c_Y)} & s(Y_{\geq 1})[1] } \]
can be completed to a $3 \times 3$ diagram of triangles. Here, $f_{\geq1}$ and $f_{\leq0}$ are the truncations of $f$ with respect to the stable $t$-structure $(\X(\geq\!1), \X(\leq\!0))$
\end{enumerate}
\end{Def}

\begin{Rem}
We require the additional axiom (\ref{nana}) from \cite{Sch}\cite[3.10]{PV}, which ensures that the realization functor is triangulated, see \cite[Appendix]{PV}.
\end{Rem}

The following result shows that the existence of an f-category over a given triangulated category assures the existence of a realization functor.
\begin{Thm}[{\cite[Section 3.1]{BBD}\cite[Appendix]{B}\cite[3.11]{PV}}]\label{realization}
Let $\T$ be a triangulated category with a $t$-structure with heart $\H$. Assume there is an f-category over $\T$. Then, the following hold.
\begin{enumerate}
\item There exists a triangle functor $F \colon \rD(\H) \to \T$ extending the inclusion $\H \subset \T$.
\item $F$ induces isomorphisms $\Hom_{\rD(\H)}(X,Y[n]) \to \Hom_\T(FX,FY[n])$ for all $X, Y \in \H$ and $n \leq 1$.
\item\label{tfae} The following are equivalent:
\begin{enumerate}
\item $F$ is fully faithful.
\item\label{eff} For any $X, Y \in \H$, $n \geq 2$, and a morphism $X \to Y[n]$ in $\T$, there exists an epimorphism $P \to X$ in $\H$ such that the composition $P \to X \to Y[n]$ is zero.
\item\label{coeff} For any $X, Y \in \H$, $n \geq 2$, and a morphism $X \to Y[n]$ in $\T$, there exists a monomorphism $Y \to I$ in $\H$ such that the composition $X \to Y[n] \to I[n]$ is zero.
\end{enumerate}
\end{enumerate}
\end{Thm}

\subsection{Filtered derived category of a DG category}
In this subsection, we formulate the filtered derived category of a DG category, and thereby proving that any algebraic triangulated category admits an f-category over itself. We refer to \cite{Ke94, Ke06} for general backgrounds on DG categories.


We follow the construction in \cite{Ke94} of the derived a DG category, together with that of filtered derived category of an abelian category to obtain the filtered derived category of a DG category.
\begin{Con}
(1)
First we construct $\CF(\A)$, the category of {\it filtered DG $\A$-modules}. The objects are finitely filtered DG $\A$-modules, that is, a DG $A$-module $X$ together with a finite descending filtration
\[ \cdots =X=F_aX \supset F_{a+1}X \supset \cdots \supset F_{b}X \supset F_{b+1}X =0= \cdots \]
of DG submodules of $X$. The morphisms of filtered DG modules are those of DG modules which preserve filtrations. 

We have an automorphism called the {\it shift functor} $[1]$ on $\CF(\A)$ which assigns to each $X \in \CF(\A)$ the DG module $X[1]$ with the filtration $F_i(X[1])=(F_iX)[1]$.

We also have an automorphism $s$ of $\CF(\A)$ called the {\it shift of filtration}, defined for each $X \in \CF(\A)$ by $sX=X$ as a DG module with the filtration $F_i(sX)=F_{i-1}X$. It naturally associates a morphism of functors $\a \colon 1_{\CF(\A)} \to s$ induced by the inclusions in filtrations.

We moreover have the ${\gr}${\it -functor} 
\[ \gr^i \colon \CF(\A) \to {\mathrm{C}}(\A), \quad X \mapsto F_iX/F_{i+1}X \]
to the category ${\mathrm{C}}(\A)$ of DG $\A$-modules for each $i \in \Z$.
A filtered DG module $X$ is {\it filtered acyclic} if each $F_iX$ is an acyclic DG module, or equivalently, each $\gr^iX$ is an acyclic DG module. We denote by $\CF_\ac(\A)$ the full subcategory of $\CF(\A)$ consisting of filtered acyclic DG $\A$-modules.

(2)
Next we construct $\KF(\A)$, the {\it filtered homotopy category} of $\A$. We denote by $\GrF(\A)$ the category of filtered graded $\A$-modules. We have the forgetful functor $\CF(\A) \to \GrF(\A)$. Endow $\CF(\A)$ with the exact structure whose conflations are short exact sequences in $\CF(\A)$ which are split in $\GrF(\A)$.
\begin{Prop}[{cf. \cite[2.2]{Ke94}}]\label{cfa}
$\CF(\A)$ with the above exact structure is a Frobenius category. 
\end{Prop}
Let $X$ be a graded $\A$-module. Following \cite[2.2]{Ke94}, we define the DG module $IX$ as follows;
\begin{itemize}
\item $(IX)(A)=X(A) \oplus (X[1])(A)$ as a graded abelian group for each $A \in \A$.
\item Make $IX$ into a graded $\A$-module by setting $a(x,y)=(ax,(da)x+(-1)^pay)$ for each morphism $a$ of degree $p$.
\item Make $IX$ into a DG $\A$-module by setting $\left( \begin{smallmatrix} 0 & 0 \\ 1 & 0 \end{smallmatrix}\right)$ as the differential.
\end{itemize}
For $X \in \GrF(\A)$, we give the filtration on $IX$ by $F_i(IX)=I(F_iX)$, which yields the functor $I \colon \GrF(\A) \to \CF(\A)$. The following lemma shows that $I$ is the right adjoint of the forgetful functor $\CF(\A) \to \GrF(\A)$.
\begin{Lem}\label{adj}
For each $X \in \CF(\A)$ and $Y \in \GrF(\A)$, we have an isomorphism
\[ \Hom_{\GrF(\A)}(X,Y) \simeq \Hom_{\CF(\A)}(X,IY). \]
\end{Lem}
\begin{proof}
For $X \in \CF(\A)$, the construction of $IX$ gives functorial conflations 
\begin{equation}\label{eqa}
X \xrightarrow{i_X} IX \xrightarrow{p_X} X[1],
\end{equation}
where $i_X={\left( \begin{smallmatrix} 1 & d \end{smallmatrix}\right)}$ and $p_X={\left( \begin{smallmatrix} -d \\ 1 \end{smallmatrix}\right)}$. Then the map $f \mapsto i_X\cdot If$ gives a desired isomorphism. 
\end{proof}
\begin{proof}[Proof of Proposition \ref{cfa}]
We deduce from the above lemma that $IY$ is injective in $\CF(\A)$. By the conflations (\ref{eqa}), we see that $\CF(\A)$ has enough injectives.
We can dually construct the projective objects $PX=X[-1] \oplus X$ and the conflations $X[-1] \to PX \to X$. 
The constructions show that the projective and the injective objects coincide.
\end{proof}

Let $\KF(\A)$ be the stable category of $\CF(\A)$, hence $\KF(\A)$ is naturally a triangulated category with suspension $[1]$. 

The shift of filtration $s$ on $\CF(\A)$ clearly takes projectives to projectives in $\CF(\A)$, and induces an automorphism, again denoted by $s$ on $\KF(\A)$ together with a natural transformation $\a \colon 1_{\KF(\A)} \to s$.

The same holds for $\gr$-functors, indeed, $\gr^i(IX)=I(\gr^iX)$. Therefore, there is a family of triangle functors $\gr^i \colon \KF(\A) \to {\mathrm{K}}(\A)$ to the homotopy category $\mathrm{K}(\A)$ of $\A$.
The subcategory of $\KF(\A)$ corresponding to $\CF_\ac(\A)$ is denoted by $\KF_\ac(\A)$, whose objects are again called the filtered acyclic DG modules. Note that $\KF_\ac(\A)$ is a thick subcategory of $\KF(\A)$.

Let $f \colon X \to Y$ be a morphism in $\KF(\A)$. The mapping cone of $f$ in $\KF(\A)$ is, the mapping cone of $f \in {\mathrm{K}}(\A)$ as a DG module, which we denote by $X \oplus Y[1]$, together with the filtration $F_i(X \oplus Y[1])=F_iX \oplus F_iY[1]$.

Let again $f \colon X \to Y$ be a morphism in $\KF(\A)$. We call $f$ a {\it filtered quasi-isomorphism} if it satisfies the following equivalent conditions:
\begin{itemize}
\item $f$ induces quasi-isomorphisms $F_iX \to F_iY$ in ${\mathrm{K}}(\A)$ for all $i$.
\item The mapping cone of $f$ is filtered acyclic.
\item $\gr^i f$ is a quasi-isomorphism in ${\mathrm{K}}(\A)$ for each $i$.
\end{itemize}

(3)
Finally we obtain $\DF(\A)$, the {\it filtered derived category} of $\A$, as the Verdier quotient of $\KF(\A)$ by $\KF_\ac(\A)$, that is, the localization of $\KF(\A)$ with respect to the filtered quasi-isomorphisms.

Since the functors $s \colon \KF(\A) \to \KF(\A) $ and $\gr^i \colon \KF(\A) \to \mathrm{K}(\A)$ take filtered quasi-isomorphisms to (filtered) quasi-isomorphisms, they induce an automorphism $s \colon \DF(\A) \to \DF(\A)$ together with a natural transformation $\a \colon 1_{\DF(\A)} \to s$, and triangle functors $\gr^i \colon \DF(\A) \to {\mathrm{D}}(\A)$ to the derived category $\mathrm{D}(\A)$ of $\A$.
\end{Con}
Set
\begin{equation}\label{eqdf}
\begin{aligned}
\DF(\A)(\geq\!0)&=\{ X \in \DF(\A) \mid \gr^iX=0 \ \text{for} \ i<0 \}, \\
\DF(\A)(\leq\!0)&=\{ X \in \DF(\A) \mid \gr^iX=0 \ \text{for} \ i>0 \}.
\end{aligned}
\end{equation}
These are thick subcategories of $\DF(\A)$. Also define the functor ${\mathrm{D}}(\A) \to \DF(\A)$ by giving the trivial filtration, that is, assign to each DG module $X$ a filtration $\{F_iX \}_i$ by $F_iX=X$ for $i \leq 0$ and $F_iX=0$ for $i>0$. This functor yields an equivalence ${\mathrm{D}}(\A) \to \DF(\A)(\geq\!0) \cap \DF(\A)(\leq\!0)$.

Now we are ready to state the following main result, whose proof is completely analogous to the case of filtered derived categories of abelian categories.
\begin{Thm}
Let $\A$ be a DG category. Then, the filtered derived category $\DF(\A)$ is an f-category over ${\mathrm{D}}(\A)$.
\end{Thm}
\begin{proof}
Let $s$ be the shift of filtration and $\a \colon 1_{\DF(\A)} \to s$ be the natural transformation constructed above. Then, these data together with the subcategories (\ref{eqdf}) and the equivalence ${\mathrm{D}}(\A) \simeq \DF(\A)(\geq\!0) \cap \DF(\A)(\leq\!0)$ defines an f-category structure. Indeed, the axioms (\ref{three}), (\ref{four}), (\ref{five}) are clear. We refer to \cite[6.3]{Sch} for (\ref{one}), (\ref{two}), (\ref{six}), and to \cite[7.4]{Sch} for (\ref{nana}).
\end{proof}

\begin{Cor}\label{algf}
Any idempotent-complete algebraic triangulated category has an f-category over itself.
\end{Cor}
\begin{proof} By \cite[4.3]{Ke94}, any idempotent-complete algebraic triangulated category is the perfect derived category $\per\A$ a DG category $\A$. By \cite[3.8]{PV}, the f-category $\DF(\A)$ over ${\mathrm{D}}(\A)$ restricts to that over $\per\A$.
\end{proof}

We summarize the important consequence in the following corollary.
\begin{Cor}\label{algreal}
Let $\T$ be an idempotent-complete algebraic triangulated category. Assume $\T$ has a $t$-structure with heart $\H$. Then, there exists a realization functor $\rD(\H) \to \T$.
\end{Cor}
\begin{proof} This is a direct consequence of Theorem \ref{realization} and Corollary \ref{algf}.
\end{proof}
\end{appendix}
\end{document}